\documentclass[british,english,american]{article}
\usepackage[T1]{fontenc}
\PassOptionsToPackage{obeyFinal}{todonotes}
\usepackage{color}
\definecolor{note_fontcolor}{rgb}{0.800781, 0.800781, 0.800781}
\usepackage{textcomp}
\usepackage{todonotes}
\usepackage{amsmath}
\usepackage{amsthm}
\usepackage{amssymb}
\usepackage{graphicx}
\usepackage{geometry}
\usepackage{comment}

\makeatletter

\numberwithin{equation}{section}
\numberwithin{figure}{section}
\theoremstyle{plain}
\newtheorem*{thm*}{\protect\theoremname}
\theoremstyle{plain}
\newtheorem{thm}{\protect\theoremname}[section]
\theoremstyle{definition}
\newtheorem{example}[thm]{\protect\examplename}
\theoremstyle{definition}
\newtheorem{defn}[thm]{\protect\definitionname}
\theoremstyle{plain}
\newtheorem{cor}[thm]{\protect\corollaryname}
\theoremstyle{plain}

\theoremstyle{plain}
\newtheorem{lem}[thm]{\protect\lemmaname}
\theoremstyle{remark}
\newtheorem{rem}[thm]{\protect\remarkname}
\theoremstyle{plain}
\newtheorem{prop}[thm]{\protect\propositionname}
\theoremstyle{plain}
\newtheorem{criterion}[thm]{\protect\criterionname}

\newtheorem{Thm}{Theorem}

\newtheorem{Cor}[Thm]{Corollary}

\@ifundefined{date}{}{\date{}}

\usepackage{hyperref}
\hypersetup{colorlinks,citecolor=blue}

\makeatother

\usepackage{babel}
\addto\captionsamerican{\renewcommand{\corollaryname}{Corollary}}
\addto\captionsamerican{\renewcommand{\criterionname}{Criterion}}
\addto\captionsamerican{\renewcommand{\definitionname}{Definition}}
\addto\captionsamerican{\renewcommand{\examplename}{Example}}
\addto\captionsamerican{\renewcommand{\factname}{Fact}}
\addto\captionsamerican{\renewcommand{\lemmaname}{Lemma}}
\addto\captionsamerican{\renewcommand{\propositionname}{Proposition}}
\addto\captionsamerican{\renewcommand{\remarkname}{Remark}}
\addto\captionsamerican{\renewcommand{\theoremname}{Theorem}}
\addto\captionsbritish{\renewcommand{\corollaryname}{Corollary}}
\addto\captionsbritish{\renewcommand{\criterionname}{Criterion}}
\addto\captionsbritish{\renewcommand{\definitionname}{Definition}}
\addto\captionsbritish{\renewcommand{\examplename}{Example}}
\addto\captionsbritish{\renewcommand{\factname}{Fact}}
\addto\captionsbritish{\renewcommand{\lemmaname}{Lemma}}
\addto\captionsbritish{\renewcommand{\propositionname}{Proposition}}
\addto\captionsbritish{\renewcommand{\remarkname}{Remark}}
\addto\captionsbritish{\renewcommand{\theoremname}{Theorem}}
\addto\captionsenglish{\renewcommand{\corollaryname}{Corollary}}
\addto\captionsenglish{\renewcommand{\criterionname}{Criterion}}
\addto\captionsenglish{\renewcommand{\definitionname}{Definition}}
\addto\captionsenglish{\renewcommand{\examplename}{Example}}
\addto\captionsenglish{\renewcommand{\factname}{Fact}}
\addto\captionsenglish{\renewcommand{\lemmaname}{Lemma}}
\addto\captionsenglish{\renewcommand{\propositionname}{Proposition}}
\addto\captionsenglish{\renewcommand{\remarkname}{Remark}}
\addto\captionsenglish{\renewcommand{\theoremname}{Theorem}}
\providecommand{\corollaryname}{Corollary}
\providecommand{\criterionname}{Criterion}
\providecommand{\definitionname}{Definition}
\providecommand{\examplename}{Example}
\providecommand{\factname}{Fact}
\providecommand{\lemmaname}{Lemma}
\providecommand{\propositionname}{Proposition}
\providecommand{\remarkname}{Remark}
\providecommand{\theoremname}{Theorem}

\global\long\def\res{\!\restriction}%
\global\long\def\N{\mathbb{N}}%
\global\long\def\Z{\mathbb{Z}}%
\global\long\def\Q{\mathbb{Q}}%
\global\long\def\R{\mathbb{R}}%
\global\long\def\C{\mathbb{C}}%
\global\long\def\T{\mathbb{T}}%

\global\long\def\bG{\mathbf{G}}%
\global\long\def\bH{\mathbf{H}}%
\global\long\def\bR{\mathbf{R}}%
\global\long\def\bS{\mathbf{S}}%
\global\long\def\bU{\mathbf{U}}%
\global\long\def\bL{\mathbf{L}}%
\global\long\def\bP{\mathbf{P}}%
\global\long\def\bN{\mathbf{N}}%
\global\long\def\bZ{\mathbf{Z}}%
\newcommand{\FC}{\operatorname{FC}}
\newcommand{\Lie}{\operatorname{Lie}}
\newcommand{\Prob}{\operatorname{Prob}}
\newcommand{\Erg}{\operatorname{Erg}}
\newcommand{\Res}{\operatorname{Res}}
\newcommand{\Ind}{\operatorname{Ind}}
\newcommand{\Ch}{\operatorname{Ch}}
\newcommand{\Tr}{\operatorname{Tr}}

\newcommand{\PD}{\operatorname{PD}}
\newcommand{\Rad}{\operatorname{Rad}}

\newcommand{\Sub}{\operatorname{Sub}}

\newcommand{\twoone}{\operatorname{II_1}}

\newcommand{\supp}{\operatorname{supp}}

\newcommand{\bary}{\operatorname{bar}}

\newcommand{\Fix}{\operatorname{Fix}}
\newcommand{\Aut}{\operatorname{Aut}}

\newcommand{\Ext}{\operatorname{Ext}}

\newcommand{\IM}{\operatorname{Im}}

\newcommand{\SL}{\operatorname{SL}}

\newcommand{\GL}{\operatorname{GL}}

\newcommand{\SO}{\operatorname{SO}}

\newcommand{\SU}{\operatorname{SU}}

\newcommand{\dd }{\,{\rm d}}

\newcommand{\ext}{\operatorname{ext}}

\begin{document}

\title{Charmenability and Stiffness of Arithmetic Groups}
\author{Uri Bader, Itamar Vigdorovich}
\maketitle

\begin{abstract}

We characterize
charmenability among arithmetic groups and deduce dichotomy statements
pertaining normal subgroups, characters, dynamics, 
representations and associated operator algebras. We do this by studying the stationary
dynamics on the space of characters of the
amenable radical, and in particular we establish stiffness: any stationary
probability measure is invariant. This generalizes a classical result
of Furstenberg for dynamics on the torus. Under a higher rank assumption, we show that any action
on the space of characters of a finitely generated virtually nilpotent group  is stiff. 
\end{abstract}

\section{Introduction}

In this paper we are concerned with the dynamics of the conjugation action of a group $\Lambda$ 
on the compact convex space $\PD_{1}(\Lambda)$ consisting all normalized positive-definite functions on $\Lambda$. 
In particular, we are interested in the structure of invariant compact convex subsets and faces of $\PD_{1}(\Lambda)$.
A special attention will be given to the subset of fixed points $\Tr(\Lambda):=\PD_{1}(\Lambda)^\Lambda$,
and the corresponding set of extreme points $\Ch(\Lambda):=\ext(\Tr(\Lambda))$.
An element of $\Tr(\Lambda)$ will be called a \emph{trace} of $\Lambda$ 
and an element of $\Ch(\Lambda)$ will be called a \emph{character}\footnote{Note that in some texts, e.g \cite{bader2022charmenability},
the term character refers to what we call here a trace.}.

\begin{defn}\label{def:charmenability}
A countable group $\Lambda$ is said to be \emph{charmenable} if it satisfies the following two properties:
\begin{enumerate}
\item (Ubiquity of traces) Every compact convex $\Lambda$-invariant subset of $\PD_{1}(\Lambda)$
contains a trace. 
\item (Character dichotomy) Every character of $\Lambda$ is either supported on $\Rad(\Lambda)$
or it is von Neumann amenable (see Definition~\ref{def:vNamenable}). 
\end{enumerate}
\end{defn}

Here $\Rad(\Lambda)$ denotes the \emph{amenable radical} of $\Lambda$,
namely its unique maximal normal amenable subgroup, and an element of $\PD_{1}(\Lambda)$ is supported on 
$\Rad(\Lambda)$ if it vanishes outside this subgroup. 
We note that amenable groups are charmenable in a strong way: they satisfy both conditions in item 2 above. 
For non-amenable groups these two conditions are exclusive.

Our main result is the characterization of charmenability among arithmetic groups,
which in this paper are defined as follows.

\begin{defn} \label{def:arithmeticity}
Given a $\Q$-algebraic group $\bG$,
a subgroup $\Lambda<\bG(\Q)$ is said to be an \emph{arithmetic subgroup of $\bG$} if 
for some (equivalently, any) faithful $\Q$-rational representation $\bG \to \GL_{n}$,
$\Lambda$ is commensurable with the preimage of $\GL_{n}(\Z)$ in $\bG(\Q)$.
\end{defn}

\begin{Thm} \label{thm:Main-charmenability} 
Let $\bG$ be a connected $\Q$-algebraic group.
Then either all arithmetic subgroups of $\bG$
are charmenable or none of them are. 
Denoting by $\bR$ the solvable radical of $\bG$,
the first case occurs if and only if the real rank of $\bG/\bR$
is not equal to $1$ and $\bG/\bR$ has at most one $\R$-isotropic $\Q$-simple 
factor. 
\end{Thm}
The following proposition describes a variety of properties that charmenable groups satisfy, most of which are proved in \cite{bader2022charmenability}. An exception is the third property which is novel (see Theorem \ref{thm:charm-dichotomy-IRS}).
The notation involved will be clarified and its proof will be discussed in \S\ref{sec:cor-of-charm}.

\begin{prop}
	\label{prop:charmenable dichotomy} 
For a charmenable group $\Lambda$ the following hold:
	\begin{enumerate}
		\item \label{charm-item-normal}Every normal subgroup is either amenable or co-amenable. 
		\item \label{charm-item-trace}Every trace is either amenable or supported on $\Rad(\Lambda)$. 
		\item \label{charm-item-IRS} For every probability measure preserving action with a spectral gap, the point stabilizers are either a.e amenable or a.e co-amenable. 
		\item \label{charm-item-URS} Every URS is either contained in $\Sub(\Rad(\Lambda))$ or it admits
		a $\Lambda$-invariant probability measure (in which case the previous item applies). 
		\item \label{charm-item-unirep}Every unitary representation is either amenable or weakly contains
		a representation which is induced from $\Rad(\Lambda)$.
		\item \label{charm-item-vn}If $M$ is a finite factor von Neumann algebra admitting a representation $\Lambda\to \mathcal{U}(M)$ with ultraweak-dense image, then either $M$ is  hyperfinite, or $M$ is isomorphic to the von Neumann algebra of a representation which is induced from $\Rad(\Lambda)$. 
		\item \label{charm-item-C*}The $C^*$-algebra $A$ of any non-amenable unitary representation of $\Lambda$ admits a unital surjective *-homomorphism  onto the $C^*$-algebra of a representation which is induced from $\Rad(\Lambda)$.   Furthermore, any maximal ideal of $A$ is the kernel of such a surjection, and any tracial state on $A$ factorizes through such a surjection. 
	\end{enumerate}
\end{prop}

When property (T) is assumed, stronger statements are obtained. For example, amenable  (co-amenable) subgroups must be finite (of finite index),  amenable representations must contain a finite dimensional subrepresentation, etc.  With further assumptions on the amenable radical we obtain:
\begin{prop}\label{prop:charmenable dichotomy prop (T)}
	Let $\Lambda$ be a charmenable group with property (T). Assume that $\Rad(\Lambda)$ has countably many subgroups. Then every  ergodic IRS is finitely supported and every URS is finite. 
\end{prop}

The amenable radical of an arithmetic group is virtually polycyclic (Corollary \ref{cor:amenable radical is virtually polycyclic}) and thus has countably many subgroups. We therefore deduce the following corollary. 
\begin{Cor} \label{cor:arithmetic dichotomy} 
Let $\bG$ be a connected $\Q$-algebraic group
and let $\Lambda<\bG(\Q)$ be an arithmetic subgroup.
Denoting by $\bR$ the solvable radical of $\bG$, if the real rank of $\bG/\bR$
is not equal to $1$ and $\bG/\bR$ has at most one $\R$-isotropic $\Q$-simple 
factor then $\Lambda$ satisfies all properties in Proposition~\ref{prop:charmenable dichotomy}. 
If moreover $\Lambda$ has property (T), then every ergodic IRS is finitely supported and every URS is finite. 
\end{Cor}

The proof of
Theorem ~\ref{thm:Main-charmenability} will be given in \S\ref{subsec:thm:main-charmenability} and the proof of Propositions \ref{prop:charmenable dichotomy} and \ref{prop:charmenable dichotomy prop (T)} are concluded in \S\ref{sec:cor-of-charm}.
Theorem~\ref{thm:Main-charmenability} in case $\bR$ is trivial, namely when $\bG$ is semisimple, was proved in 
\cite{boutonnet2021stationary} and \cite{bader2022charmenability}.
Our main concern in this paper is dealing with the case where $\bR$ is non-trivial.
In that case \cite[Proposition~4.19]{bader2022charmenability} enables us to reduce the proof of Theorem~\ref{thm:Main-charmenability}
to a certain stiffness result.
A prototypical example is the proof of \cite[Theorem~C]{bader2022charmenability} which 
establishes the charmenability of $\SL_n(\Z)\ltimes \Z^n$
using Furstenberg's classical theorem \cite{furstenberg1998stiffness} regarding the stiffness of the standard action of $\SL_{n}(\Z)$ on $\T^{n}$  .
The stiffness result that we need is Theorem~\ref{thm:stiffness of algebraic groups} below.
In order to state it, we need some preparation.

\begin{defn}
Let $\Lambda$ be a countable group acting on a topological space $X$ by homeomorphisms.
Given a probability measure $\mu$ on $\Lambda$,
we say that the $\Lambda$-action on $X$ is \emph{$\mu$-stiff} if any $\mu$-stationary probability
measure\footnote{In this paper, all measures on topological spaces are assumed to be Borel.} $\nu$ on $X$, that is $\nu$ for which $\mu*\nu:=\sum_{g\in\Lambda}\mu(g)g.\nu=\nu$,
is $\Lambda$-invariant.
\end{defn}

We note that since we are not assuming $X$ to be compact, there may be no $\mu$-stationary measures at all on $X$,
in which case the action is stiff in a vacuous sense. 
This is never the case in the type of dynamical systems we consider. 

\begin{prop}\label{Char-kakutani}
Let $\Lambda$ be a countable group acting by automorphisms on a countable group $\Gamma$.
Then for every probability measure $\mu$ on $\Gamma$, the set of $\mu$-stationary probability measures on $\Ch(\Gamma)$ is non-empty.
\end{prop}

Note that $\Ch(\Gamma)$ is Polish, but may be non-compact (see Example \ref{exa:heisenberg} bellow).
The proof of Proposition~\ref{Char-kakutani} will be given in \S\ref{subsec:Basics}.

A general group has no canonical probability measure, and stiffness may very well depend on different choices of $\mu$. 
However, for arithmetic groups we have preferred measures.

\begin{defn} \label{defn:Furstenberg-arithmetic}
Let $\bG$ be an $\R$-algebraic group, $\bP<\bG$ be a minimal $\R$-parabolic subgroup 
and $\bN<\bG$ be the normalizer of $\bP$
(note that $\bP$ contains the solvable radical of $\bG$, $\bP<\bN$ is of finite index and if $\bG$ is connected then $\bP=\bN$).
A probability measure $\mu$ on $\bG(\R)$ (not necessarily fully supported)
is called a \emph{Furstenberg measure} if in the $\bG(\R)$-invariant measure class on $\bG(\R)/\bN(\R)$ 
there exists a probability measure with which this space is a Furstenberg-Poisson boundary for $(\bG(\R),\mu)$.
If $\Lambda<\bG(\R)$ is a closed subgroup which is generated as a semigroup by the support of the Furstenberg measure $\mu$,
we will say that $\mu$ is a \emph{Furstenberg measure on $\Lambda$}.
\end{defn}

Furstenberg proved that lattices in semisimple groups carry Furstenberg measures, see Proposition~\ref{prop:Fmeasures}
for an elaboration.
The following is a corollary of the above fact.

\begin{prop} \label{prop:furstmeasures-arithmetic}
For every $\Q$-algebraic group $\bG$ and an arithmetic subgroup $\Lambda$,
there exists a Furstenberg measure on $\Lambda$.
\end{prop}

The proof of Proposition~\ref{prop:furstmeasures-arithmetic} will be given in \S\ref{sec:FM}.
We are now ready to state our main stiffness result. 

\begin{Thm} \label{thm:stiffness of algebraic groups}
For an arithmetic group $\Lambda$ endowed with a Furstenberg measure $\mu$,
the action of $\Lambda$ on $\Ch(\Rad(\Lambda))$ is $\mu$-stiff.
\end{Thm}

Note that this theorem, unlike Theorems~\ref{thm:Main-charmenability} and \ref{cor:arithmetic dichotomy},
involves no rank assumption.
A main ingredient in the proof of 
Theorem~\ref{thm:stiffness of algebraic groups} is Theorem~\ref{thm:main-stiffness}
which regards stiffness of the action of lattices in a semisimple Lie group on the space of characters of a nilpotent group.  
Another consequence of the latter theorem, which might be of independent interest, is the following.

\begin{Thm}
\label{thm:higher-rank-stiffness}
Let $\Lambda$ be an irreducible lattice in a connected semisimple Lie group $G$ with finite center and let
$\mu$ be a Furstenberg measure on $\Lambda$. Assume that $G$ is
not isogenous to a group of the form $\SO(1,m)\times K$ or $\SU(1,m)\times K$,
with $K$ compact. Then for any action of $\Lambda$ by automorphisms
on a finitely generated virtually nilpotent group $\Gamma$, the corresponding
action on $\Ch(\Gamma)$ is $\mu$-stiff. 
\end{Thm}
The proofs of Theorem~\ref{thm:stiffness of algebraic groups}, Theorem~\ref{thm:higher-rank-stiffness} and Theorem~\ref{thm:main-stiffness} 
will be given in \S\ref{sec:Stiffness results}.
If $\Gamma$ is a finitely generated abelian group then $\Ch(\Gamma)$, which coincides with $\hat{\Gamma}$, the Pontryagin dual of $\Gamma$, 
is a finite union of tori.
However, Theorems~\ref{thm:stiffness of algebraic groups}, \ref{thm:higher-rank-stiffness} and \ref{thm:main-stiffness}
are not to be confused with 
the results of Benoist-Quint regarding stiffness for actions on
nilmanifolds \cite{benoist2013stationary}. 
Indeed, the space $\Ch(\Gamma)$ for a finitely generated virtually nilpotent group which is not virtually abelian is
neither homogeneous nor a manifold. 

\begin{example} \label{exa:heisenberg}
The space of characters $\Ch(H)$ of the discrete Heisenberg group
\[ H=\left\{ \left(\begin{array}{ccc}
1 & x & z\\
0 & 1 & y\\
0 & 0 & 1
\end{array}\right):x,y,z\in\Z\right\} \]
is homeomorphic to the subset
of $\mathbb{R}^{3}$ illustrated in the following picture:
\begin{figure}[!tph]
\begin{centering}
\includegraphics[clip,scale=0.6]{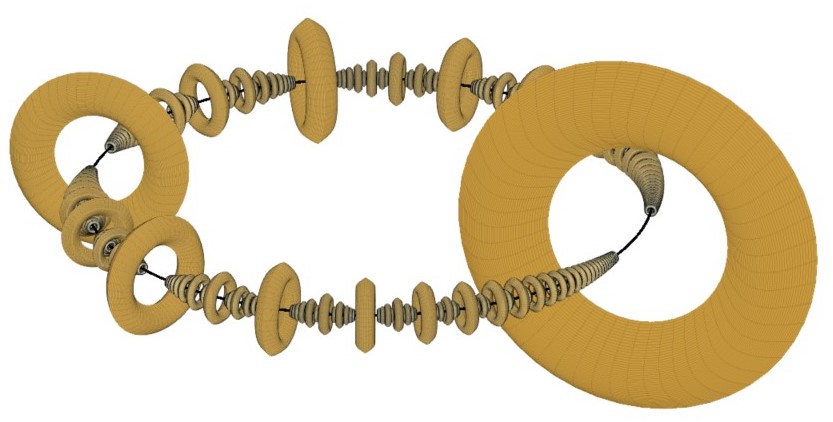}
\par\end{centering}
\end{figure}

\noindent Here points on the horizontal circle represents the space of characters of the 
center $Z<H$.
The inclusion $Z\subseteq H$ induces
a surjective continuous map $\Ch(H)\to\Ch(Z)\simeq \hat{Z}\cong\R/\Z$ (see Lemma
\ref{lem:restriction-of-characters}). 
The fiber over an irrational
point of $\R/\Z$ consists of a single point, whereas the fiber
of a reduced fraction $\frac{p}{q}\in \R/\Z$ 
could be identified with $\R^{2}/q^{-1}\Z^{2}$,
the Pontryagin dual of $q\Z^{2}$ which could be identified with 
the center of $H/qZ$ (see Theorem~\ref{thm:central-induction}).
We note that $\Ch(H)$ is not compact: the rational points of the horizontal circle in the picture above are excluded.   
\end{example}

\subsection{Related works}

An excellent introduction to the theory of characters is \cite{bekka2019unitary}. A comprehensive study of characters of nilpotent groups was carried since the 1970's, see e.g 
\cite{howe1977representations, kaniuth1980ideals, carey-moran1984nil-char, kaniuth2006induced}. This study was recently generalized to polycyclic groups, and thus to arithmetic subgroups of solvable algebraic groups \cite{levit2022characters}.

A classification of the characters of arithmetic subgroups of higher rank semisimple algebraic groups was completed
in \cite{peterson2014character} following \cite{bekka2007operator}.
The situation in the rank one case is complicated and somewhat less understood, see e.g \cite{poulsen22freegroup}.
We know of no reference dealing with character theory of general arithmetic groups,
though the paper \cite{bekka2020characters}
is closely related.

Dynamical aspects of spaces of positive-definite functions on lattices in semisimple groups 
were considered recently in \cite{boutonnet2021stationary}, \cite{bader2022charmenability} and \cite{bader2021charmenability}, see also \cite{houdayer2021noncommutative}.
The notion of ``charmenability'' is highlighted in \cite{bader2022charmenability}.
Dynamical questions regarding spaces of measures on tori are classical.
In particular, the question of ``stiffness'' is dealt with in \cite{BFLM2011stiffness} and \cite{benoist2013stationary},
following the pioneering work \cite{furstenberg1998stiffness}.
Surprisingly, it seems that questions regarding dynamics on the spaces of characters
of non-abelian nilpotent groups were never studied before.
For our main goal in this paper, proving Theorem~\ref{thm:Main-charmenability},
the stiffness result of Theorem~\ref{thm:stiffness of algebraic groups} is sufficient. 
However, it seems to us that the latter theorem is not optimal,
as we deal only with Furstenberg measures,
and we hope to improve it in a future work.
Another aspect in which we hope to improve this work in the future is by considering $S$-arithmetic groups.
This will force us to tackle characters on nilpotent groups which are not finitely generated, but we do not foresee a major problem in doing so.

\subsection{The structure of this paper}

Sections \S\ref{sec:Stationary} and \S\ref{sec:Characters} are devoted to introducing the needed preliminaries 
and basic results regarding stationary dynamics and characters correspondingly. 
In \S\ref{sec:Stiffness results} we state and prove Theorem~\ref{thm:main-stiffness} and use it to prove 
Theorems~\ref{thm:stiffness of algebraic groups} and \ref{thm:higher-rank-stiffness}.  
In \S\ref{sec:Charmenability} we prove Theorem \ref{thm:Main-charmenability}.
 \S\ref{sec:cor-of-charm} is devoted to proving Propositions~\ref{prop:charmenable dichotomy} and \ref{prop:charmenable dichotomy prop (T)}.

\subsection{Acknowledgments}
This work was supported by the ISF Moked 713510 grant number 2919/19.
The second author would like thank Barak Weiss for useful discussions and references. 
We are grateful to Tamuz Ofer for the figure in Example~\ref{exa:heisenberg}.
Both authors thank the members of the Midrasha on Groups at the Weizmann Institute for their support, friendship and 
professional encouragement.

\section{Stationary dynamics} \label{sec:Stationary} 
In this section we define and discuss general statements regarding stationary dynamics associated to a group $G$. For almost all of our purposes it suffices to consider the case where $G$ is countable. However, as some of our claims might be of independent interest, we shall mostly work in the setting where $G$ is a general locally compact second countable group. 

\subsection{Ergodic theoretical preliminaries}
By a \emph{Borel space} we mean a set endowed with a $\sigma$-algebra. Morphisms between Borel spaces are just measurable maps, and they are called \emph{Borel maps}.
A \emph{standard Borel space} is a Borel space which is isomorphic
as such to the underlying Borel space of a Polish space.   By
a \emph{Lebesgue space} we understand a standard Borel space endowed
with a measure class. A \emph{Lebesgue map} from a Lebesgue space to a Borel space is a
class of a.e equal Borel maps, defined on a conull subset of the domain.
Morphisms between Lebesgue spaces are measure class preserving Lebesgue maps.
We note that the category of Lebesgue spaces is equivalent to the opposite  category of abelian von Neumann subalgebras of bounded operators on  separable Hilbert spaces, by the functor
$X\mapsto L^\infty(X)\subseteq \mathcal{B}(L^2(X))$.

Let $G$ be a locally compact second countable group. An action of $G$ on a Borel space is \emph{Borel} if the corresponding action map is Borel. 
A\emph{ Borel
$G$-space} is a Borel space
endowed with a Borel action of $G$.  A \emph{Lebesgue $G$-space} is a Lebesgue space endowed
with a Borel $G$-action that preserves the measure class. The Lebesgue $G$-space is called
\emph{ergodic} if every measurable $G$-invariant subset is either null
or conull. 

For a topological space $X$ we denote by $\Prob(X)$ the space of all Radon probability  measures endowed with the weak-* topology. If $X$ is Polish (resp. compact) then so is $\Prob(X)$, and all Borel probability measures are Radon. Suppose $X$ and $Y$ are Polish spaces and that $p:X\to Y$ is a
continuous surjective map. Let $\eta\in \Prob(Y)$ and consider a Lebesgue map $Y\to \Prob(X)$ that takes $y$ to a measure $\nu_y$ supported on the fiber $p^{-1}(y)$. Then one may integrate these $\nu_{y}$
into a probability measure $\nu=\int_{Y}\nu_{y}d\eta$.
That is, for a bounded continuous function $f$ on $X$,
$\int_{X}f\dd\nu=\int_{Y}\left(\int_{p^{-1}(y)}fd\nu_{y}\right)\dd\eta$.
The \emph{disintegration theorem} states that every $\nu\in \Prob(X)$ is given by an integration as above w.r.t a unique Lebesgue map  $Y\to \Prob(X):y\mapsto\nu_{y}$. This is called the \emph{disintegration of $\nu$ along $p$}.

Suppose $X$ and $Y$ are Polish $G$-spaces and that $p:X\to Y$
is a continuous and surjective $G$-equivariant map.
Consider $\nu\in \Prob(X)$ whose measure class is preserved by $G$, and the pushforward measure $p_{*}\nu$ on $Y$. Then $X$
is said to be \emph{measure preserving relative to $Y$} if the disintegration
map $Y\to \Prob(X)$ associated with $\nu$ is a $G$-equivariant.

\subsection{Choquet simplices and stationary dynamics}

A general reference for Choquet Theory is \cite{phelps2001lectures}.
Recall that any compact convex set $C$ in a locally convex topological vector space is 
the closed convex hull of its extreme points $\Ext(C)$, by the Krein-Milman theorem.
Assuming that $C$ is metrizable, $\Ext(C)$ is a $G_\delta$ subset of $C$, thus a Polish space.
For a probability measure $\nu$ on $\Ext(C)$ we define the barycenter, $\bary(\nu)$, to be the unique
point $c\in C$ such that $\phi(c)=\int \phi \dd\nu$ for every continuous linear functional on the ambient space.
We obtain an affine continuous map $\bary:\Prob(\Ext(C))\to C$.
Choquet's theorem says that this map is surjective.

\begin{defn}
The metrizable compact convex set $C$ is said to be 
a \emph{Choquet simplex} if the map $\bary$ is bijective.
\end{defn}

For a metrizable Choquet simplex $C$, the space $\Prob(\Ext(C))$ is a Polish space,
thus the map $\bary$ is a Borel isomorphism by Suslin's theorem.
It thus makes sense to consider
the category whose objects are Polish convex spaces and its morphisms are Borel affine maps.
The space $\Prob(X)$, for a Polish space $X$, is a prominent example and a
Polish convex space is called a \emph{Polish simplex} if it is Borel affine isomorphic to 
a space of the form $\Prob(X)$.
Note that a Polish simplex is a Choquet simplex iff it is compact.

Given a Polish space $X$ and  a continuous action of topological group $G$ by homeomorphisms,
we have the standard identification $\Erg_G(X)=\Ext(\Prob(X)^G)$,
where $\Erg_G(X)$ denote the space of ergodic $G$-invariant probability measures on $X$.
Integration gives a continuous affine map
$\Prob(\Erg_G(X)) \to \Prob(X)^G$,
and this map is a bijection by the existence and uniqueness of the disintegration map,
considering for a given $\nu\in \Prob(X)^G$ the $\nu$-a.e defined Borel map dualizing the von Neumann algebra inclusion 
$L^\infty(X,\nu)^G<L^\infty(X,\nu)$. 
Thus $\Prob(X)^G$ is a Polish simplex.

\begin{prop} \label{prop:choqueinv}
Let $C$ be a metrizable Choquet simplex and assume the group $G$ acts continuously on $C$ by affine transformations.
then $C^G$ is a metrizable Choquet simplex.
\end{prop}

\begin{proof}
Denote $X=\Ext(C)$.
Then the Borel affine $G$-equivariant isomorphism $\bary:\Prob(X)\to C$ induces a
Borel affine isomorphism $\Prob(X)^G\to C^G$.
Since $\Prob(X)^G$ is a Polish simplex, it follows that $C^G$ is a Polish simplex.
It is a Choquet simplex as it is compact.
\end{proof}
For another proof of Proposition~\ref{prop:choqueinv}, see \cite[Corrolary~12.13]{kennedy2021noncommutative}.

\smallskip
Next, we provide an analogue which regards stationary measures.
We start by giving some background. 
Consider a continuous (strong operator topology) action of locally compact second countable group $G$ on a Banach space $E$ by linear isometries. This induces a continuous action on the dual space $E^*$ endowed with the weak-* topology. We define the convolution between a probability measure $\mu$ on $G$ and an element $\varphi \in E^*$ by:
\begin{align}\label{convolution} (\mu * \varphi) (v):= \int_{G}  \varphi(g^{-1}.v)\dd\mu(g) \hspace{20 pt} (v\in E) 
\end{align}
It is easy to see that this integral is absolutely convergent, and that $\mu*\varphi$ is again an element of $E^*$. We say that $\varphi$ is \emph{$\mu$-stationary} if $\mu*\varphi = \varphi$. 
Equation \ref{convolution} thus defines a continuous action of the topological semigroup $\Prob(G)$, with the convolution operation and the weak-* topology, on $E^*$. We are interested in $\Prob(G)$-invariant subsets $C\subseteq E^*$.  
If $C$ is convex, weak-* compact and $G$-invariant, then $C$ is definitely $\Prob(G)$-invariant. Indeed the convolution operation is given by the following composition: 
\[ \Prob(G)\times C \to \Prob(G \times C) \to \Prob(C)  \to C, \quad (\mu,c) \mapsto (\mu\times \delta_c) \mapsto 
a_*(\mu\times \delta_c) \mapsto \bary(a_*(\mu\times \delta_c)).\]
where $a:G\times C\to C$ is the action map. 
Another important example is when $C=\Prob(X)$ where $X$ is a Polish space and the action on $C$ comes from a continuous action of $G$ on $X$. In that case $C$ is a convex subset of the dual of the Banach space of all bounded continuous functions on $X$ and the convolution operation is given by the pushforward of the action map: 
\[
\Prob(G)\times\Prob(X)\hookrightarrow \Prob(G\times X)\to \Prob(X)
\]

\begin{defn}
	Let $G$ be a locally compact group. 
	$\mu\in \Prob(G)$ is said to be \emph{admissible} if  it is absolutely continuous with respect to the Haar measure class on $G$, and it is not supported on a closed sub-semigroup of $G$. 
\end{defn}

It is well known and easy to prove that, for $\mu$ admissible, any $\mu$-stationary probability measure has a $G$-invariant class.

\begin{lem}
\label{lem:support-of-stationary-is-invariant and ergodic-iff-extremal}
Let $X$ be a Polish space and $G$ a locally compact second countable group that acts on it 
continuously by homeomorphisms.
Fix an admissible probability measure $\mu$ on $G$. Then a  $\mu$-stationary probability measure $\nu$ on $X$  is ergodic if and only if
it is an extreme point of $\Prob(X)^{\mu}$.
\end{lem}

\begin{proof}
This is \cite[Corollary 2.7]{bader2006factor}. 
In this reference, $X$ is assumed compact, but the proof works just as well if $X$ is Polish. 
\end{proof}

We denote by $\Erg_\mu(X)\subset \Prob(X)^\mu$ the subset of ergodic measures.
In view of Lemma~\ref{lem:support-of-stationary-is-invariant and ergodic-iff-extremal},
we have that the integration map
$\Prob(\Erg_\mu(X)) \to \Prob(X)^\mu$,
establishes a continuous affine bijection, again by 
the existence and uniqueness of the disintegration map,
considering for a given $\nu\in \Prob(X)^\mu$ the $\nu$-a.e defined Borel map dualizing the von Neumann algebra inclusion 
$L^\infty(X,\nu)^\mu <L^\infty(X,\nu)$. 
Thus $\Prob(X)^\mu$ is a Polish simplex.

\begin{lem}\label{lem:choquet-kakutani}
Let $G$ act on a metrizable Choquet simplex $C$ continuously by affine homeomorphisms. 
Let $X$ denote the set of extreme points of $C$ and consider the induced action of $\Prob(G)$ on $\Prob(X)$. 
Then the barycenter map $\Prob(X)\to C$ is a $\Prob(G)$-equivariant continuous bijection
and thus induces a continuous bijection of non-empty spaces $\Prob(X)^\mu \to C^\mu$, for any $\mu\in \Prob(G)$.
\end{lem}

\begin{proof}
The map $\Prob(X)\to C$ is a continuous affine bijection.
It is obviously $G$-equivariant, thus also $\mu$-equivariant for any $\mu\in \Prob(G)$.
Therefore it induces a bijection $\Prob(X)^\mu \to C^\mu$.
By Kakutani's fixed point theorem, $C^\mu$ is non-empty.
It follows that $\Prob(X)^\mu$ is non-empty as well.
\end{proof}

Arguing as in the proof of Proposition~\ref{prop:choqueinv},
we deduce the following.

\begin{prop} \label{prop:choquestat}
Let $C$ be a metrizable Choquet simplex and assume the locally compact second countable group $G$ acts on $C$ continuously by affine transformations.
Let $\mu$ be an admissible probability measure on $G$ and consider its convolution action on $C$.
Then $C^\mu$ is a metrizable Choquet simplex.
\end{prop}

\subsection{Stiffness} \label{sec:stationary}
We retain the setting of a locally compact second countable group $G$ acting on a Banach space $E$ by linear isometries, along with the corresponding continuous action of the semigroup $\Prob(G)$ on $E^*$ as given in Equation (\ref{convolution}). 

\begin{defn}\label{def:affine-stiff}
Let $C$ be a $\Prob(G)$-invariant subset of $E^*$. We say that the action of $G$ on $C$ is \emph{affinely $\mu$-stiff} for a probability measure $\mu\in \Prob(G)$, if any $\mu$-stationary point in $C$ is $G$-invariant.  A continuous action of $G$ on a Polish space $X$ is said to be \emph{topologically $\mu$-stiff} if the corresponding action on $\Prob(X)$ is affinely $\mu$-stiff.
\end{defn}

When there is no risk of confusion, e.g when a Polish $G$-space
$X$ does not admit any a priori affine structure, we shall simply
say \emph{$\mu$-stiff}. We shall also often say that a Polish space
$X$ is stiff as long as the action of $G$
as well as the probability measure $\mu\in \Prob(G)$ are implied by context.

We remark that in Definition \ref{def:affine-stiff} it could very well be the case that $C^\mu$, the set of $\mu$-stationary points in $C$, is  empty, in which case the action is $\mu$-stiff in a vacuous sense. 
This is never the case when $C$ is weak-* compact due to the standard Kakutani argument. 
By Lemma~\ref{lem:choquet-kakutani},
this is also not the case for $\Prob(X)$, where $X=\Ext(C)$ for some metrizable Choquet simplex $C$.

\begin{lem}
	\label{ergodic-decomposition}
	Let $\mu\in \Prob(G)$  admissible. The action of $G$ on a Polish space $X$ is $\mu$-stiff if and only if any ergodic $\nu\in \Prob(X)^\mu$ is invariant.  
\end{lem}

\begin{proof}
	This is an immediate application of Lemma~\ref{lem:support-of-stationary-is-invariant and ergodic-iff-extremal} and the fact that $\Prob(X)^\mu$ is a Polish simplex.
\end{proof}

We shall now state a few general and simple lemmas on stiffness.

\begin{lem} \label{lem:finite-is-stiff}
Assume $G$ is a finite group and $\mu$ is an admissible probability measure on $G$. 
Then every $G$-invariant convex subset $C\subset E^*$ is affinely $\mu$-stiff.
\end{lem}

\begin{proof}
Observe that $\Prob(G)^\mu=\{\lambda\}$, where $\lambda$ is the uniform measure on $G$. 
As a result $1/n\sum_{k=1}^n\mu^k$ converges to $\lambda$ and therefore $C^\mu = C^\lambda =C^G$.
\end{proof}

\begin{lem}
	\label{lem:stationarity-of-quotients} Let $C$ be a convex $\Prob(G)$-invariant subset of $E^*$. Let $\bar{G}$ be the quotient of $G$ by the kernel of $G\curvearrowright C$ and let  $\bar{\mu}$
	be the pushforward of $\mu$ under the quotient map $G\to\bar{G}$.
	Then $C$ is affinely $\mu$-stiff if and only if it is affinely $\bar{\mu}$-stiff. 
\end{lem}

\begin{proof}
	Since the $G$-action on $C$ factorizes through $\bar{G}$, it is easy to see that the corresponding $\Prob(G)$-action on $C$ factorizes through $\Prob(\bar{G})$. This shows that $C^\mu=C^{\bar{\mu}}$. As it is clear that $C^G=C^{\bar{G}}$, the statement follows. 
\end{proof}

\begin{lem}
	\label{lem:factor-of-stiff-is-stiff} Let $p:C_1\to C_2$ be a continuous affine surjective
	$G$-equivariant map between two $\Prob(G)$-invariant convex sets and assume that the fibers of $p$ are compact. If $C_1$ is affinely $\mu$-stiff then so is $C_2$. 
\end{lem}

\begin{proof}
	For  $c_2\in C_2^\mu$, the fiber $p^{-1}(c_2)$ is non-empty, compact, convex and closed under convolution by $\mu$. Thus we can choose a $\mu$-stationary element $c_1\in p^{-1}(c_2)$. Since $C_1$ is stiff, $c_1$ must be $G$-invariant, and as a result, $c_2=p(c_1)$ is $G$-invariant.
\end{proof}

\subsection{Relative stiffness}
We introduce a relative version of stiffness for actions of  a locally compact second countable group $G$ endowed with an admissible probability measure $\mu$. 
\begin{defn}
	Let $p:X\to Y$ be a $G$-equivariant continuous map between
Polish $G$-spaces. $X$ is said to be\emph{ $\mu$-stiff relative
to $Y$} if $X$ is measure preserving relative to $Y$ with respect
to any choice of an ergodic $\mu$-stationary probability measure $\nu$ on
$X$.
\end{defn}

\begin{lem}
\label{lem:relaltively-stiff} If $Y$ is
stiff and $X$ is stiff relative to $Y$ then $X$ is stiff. 
Conversely, if $X$ is stiff then it is stiff relative to $Y$ and,
assuming further that the fibers of $p$ are compact, $Y$ is stiff.
\end{lem}

\begin{proof}
Fix some $\nu\in \Prob(X)^{\mu}$ ergodic, and let $\nu=\int\nu_{y}dp_{*}\nu$
be the disintegration of $\nu$ along $p$. 

If $X$ is stiff relative to $Y$ and $Y$ is stiff then
$p_{*}\nu$ is $G$-invariant, and: 
\[
g.\nu=\int_{Y}g.\nu_{y}dp_{*}\nu=\int_{Y}\nu_{g.y}dp_{*}\nu=\int_{Y}\nu_{y}dp_{*}\nu=\nu\,\,\,\,\,(\forall g\in G)
\]
thus $\nu$ is $G$-invariant. This shows $X$ is stiff in light of Lemma \ref{ergodic-decomposition}.

The fact that if $X$ is stiff then it is stiff relative to $Y$ follows from the uniqueness of the disintegration, 
as $\nu=\int_{Y}g.\nu_{g^{-1}y}dp_{*}\nu$ implies $\nu_{y}=g.\nu_{g^{-1}y}$ a.e.
Further, if the fibers of $p$ are compact and $X$ is stiff then $Y$ is stiff by Lemma \ref{lem:factor-of-stiff-is-stiff}.
\end{proof}

We introduce two cases in which relative stiffness is implied. 

\begin{prop} \label{prop:countable-fibers-stiffness}
Let $p:X\to Y$ be a $G$-equivariant
continuous map between Polish spaces and assume all fibers are countable
(or finite).  Then $X$ is stiff relative to $Y$. 
\end{prop}
This generalizes the well known fact that actions on countable spaces are stiff. 
\begin{proof} 
Let $\nu\in \Prob(X)^{\mu}$ ergodic and denote $\eta=p_*\nu$. 
Disintegrate $\nu$ along $p$, namely write $\nu=\int_{Y}\nu_{y}\dd \eta$
where $\nu_{y}\in \Prob(X)$ is defined for $\eta$-a.e $y$ and is supported
on the fiber $p^{-1}(y)$. For each $y\in Y$ for which $\nu_{y}$
is defined, let $M_{y}$ be the (finite) set of all $x\in p^{-1}(y)$
that attain a maximal $\nu_{y}$-measure (recall that $p^{-1}(y)$
is countable). Denote by $M$ the union of all such $M_{y}$, and suppose first that we have managed to show that $\nu(M)=1$. In that case,  $\nu_y$ is the uniform measure supported on $M_y$ for a.e $y$. The fact that $G$ preserves the measure class of $\nu$ implies that $M$ is essentially $G$-invariant. Thus writing $\nu$ according to its integration formula we get that for all $g\in G$,  $\nu_y(M_y\cap gM_{g^{-1}y})=1$ for a.e $y\in Y$, or equivalently $g\nu_{g^{-1}y}=\nu_y$.  By Fubini,  for a.e $y\in Y$, the equality $g\nu_{g^{-1}y}=\nu_y$ holds for a.e $g\in G$. This equality must hold in fact for all $g\in G$ due to the continuity of the action on the space of Lebesgue maps $L(Y,\Prob(X))$ in the weak-$*$ topology induced from its containment in $L^1(Y,C_b(X))^*$. Relative stiffness therefore follows.

In order to show that indeed $\nu(M)=1$, choose some Borel section $s:Y\to M$ of $p\res_{M}$. The image $s(Y)$ is a Borel subset of $X$ (Suslin's theorem), and its $\nu$-measure is positive (as its $\nu_y$ measure is positive for a.e $y\in Y$).  For any $g\in G$ denote by $s^{g}:Y\to X$ the section
defined by $s^{g}(y)=gs(g^{-1}y)$. By the way $M$ was constructed, it holds that
$\nu_y\left(s(y)\right)\geq\nu_y\left(s^{g}(y)\right)$, where equality holds if and only if $s^g(y)\in M$. In particular $\nu(s(Y))\geq\nu(s^{g}(Y))$, and equality holds if and only if $s^g(Y)\subseteq M$, up to null sets. The equality $\nu(s(Y))=\nu(s^{g}(Y))$   indeed holds for $\mu$-a.e $g\in G$ due to the stationarity of $\nu$:
\[
\nu(s(Y))=\int_G\nu\left(g.s(Y)\right)\dd\mu=\int_G\nu\left(s^{g}(Y)\right)\dd\mu
\]
We see that for Haar-a.e $g\in G$,  $s^g(Y)\subseteq M$ up to null sets, 
because $\mu$ is admissible. This must  hold for all $g\in G$ due to the continuity of the $G$-action on $L^\infty(X,\nu)$ with the weak-* topology, along with the observation that the set of all $f\in L^\infty(X,\nu)$ for which $\chi_M\cdot f=f$ is closed. 

Fix a countable dense subgroup $G_0\leq G$. 
The above shows that  the set $M_0:=\bigcup_{g\in G_0} s^g(Y)$ is contained in $M$, up to null sets. But $M_0$ has positive measure as it contains $s(Y)$, and it is clearly $G_0$-invariant.
By the ergodicity of $\nu$ we conclude that $M_0$ has full $\nu$-measure. In particular $M$ is of full measure, as desired. 
\end{proof}

The second setting where we show relative stiffness is for extension by compact groups:
\begin{prop} \label{prop:X->X/K is relatively stiff}Let $X$ be a Polish $G$-space
and let $K$ be a compact group acting continuously on $X$ and commuting
with the $G$-action. Then $X$ is stiff relatively to $X/K$.
\end{prop}
The proof uses Zimmer's notion of minimal cocycles which we now introduce (see \cite{zimmer2020extensions} for further details). 
Let $\Omega$ be an ergodic Lebesgue $G$-space and let $K$ be a
second countable locally compact group. A \emph{cocycle of $G$
	over $\Omega$ taking values in $K$} is a Lebesgue map $\alpha:G\times \Omega\to K$
satisfying:
\[
\alpha(g_{1}g_{2},\omega)=\alpha(g_{1},g_{2}\omega)\alpha(g_{2},\omega);\,\,\,\,\,\left(g_{1},g_{2}\in G,\,\omega\in \Omega \right)
\]
Given $\alpha$ one can construct the Lebesgue $G$-space $\Omega\times_{\alpha}K$
defined as follows: as a Lebesgue space it is just $\Omega\times K$
with the product $\sigma$-algebra and with the product of the measure
class of $\Omega$ with the Haar measure class on $K$. The $G$-action
on this space is defined by: $g.\left(\omega,k\right)=(g.\omega,\alpha(g,\omega)k)$.
It is not hard to see that this turns $\Omega\times_{\alpha}K$ into a
Lebesgue $G$-space.
Two cocycles $\alpha$ and $\beta$ are said to be \emph{cohomologous}
if there exists an a.e defined measurable map $\varphi:\Omega\to K$ for
which:
\[
\beta(g,\omega)=\varphi(g.\omega)^{-1}\alpha(g,\omega)\varphi(\omega)
\]
for all $g\in G$ and for a.e $\omega\in \Omega$. In such a case $\Phi:(\omega,k)\mapsto(\omega,\varphi(\omega)k)$
defines an isomorphism of Lebesgue $G$-space between $\Omega\times_{\alpha}K$
and $\Omega\times_{\beta}K$.

 While $\alpha$ takes values in $K$, it
could very well be that $\alpha$, or perhaps some cocycle $\beta$
cohomologous to $\alpha$, takes values in a proper subgroup of $K$.
To that end define $K_{\alpha}:=\overline{\left\langle \IM(\alpha)\right\rangle }$,
that is the minimal closed subgroup of $K$ containing $\left\{ \alpha(g,\omega)\right\} _{g\in G,\omega \in \Omega}$.
$\alpha$ is called a \emph{minimal cocycle}, if there is no $\beta$
cohomologous to $\alpha$ for which $K_{\beta}\subsetneq K_{\alpha}$.
\begin{thm}[\protect{\cite[Corollary 3.8]{zimmer2020extensions}}]
	\label{thm:minimal-cocycle}Let $\Omega$ be an ergodic Lebesgue $G$-space,
	and let $K$ be a second countable \emph{compact} group. Then:
	\begin{enumerate}
		\item \label{enu:minimal-cocycle-existence}Any cocycle $\alpha:G\times \Omega\to K$
		is cohomologous to a minimal cocycle.
		\item If $\alpha$ and $\beta$ are cohomologous minimal
		cocycles then $K_{\alpha}$ and $K_{\beta}$ are conjugated in $K$.
		\item \label{enu:minimal-cocycle-ergodicity} $\alpha$ is a minimal cocycle
		with $K=K_{\alpha}$ if and only if $\Omega\times_{\alpha}K$ is ergodic.
	\end{enumerate}
\end{thm}

\begin{proof}[Proof of Proposition \ref{prop:X->X/K is relatively stiff}]
Let $\nu\in \Prob(X)^{\mu}$ ergodic.
We also assume as we may, upon restriction, that $X=K\cdot\supp{\nu}$.
We consider the Borel map $X\to \Sub(K)$ taking each $x$ to its
stabilizer $K_{x}$ in $K$. Since the action of $G$ commutes
with that of $K$, this map is $G$-invariant,
hence, as $\nu$ is ergodic, essentially constant on some subgroup
$K_{0}$. 
Moreover, this map is $K$-equivariant when $K$ acts on
$\Sub(K)$ by conjugation, so $K_{0}$ is normal
in $K$. 
By $X=K\cdot\supp{\nu}$ we get that $K_0$ acts trivially on $X$, thus we assume as we may that $K_0$ 
is trivial. Replacing $X$ with a conull subset which is $G\times K$ invariant,
we assume as we may that $K$ acts freely on $X$. 

Denote $Y=X/K$ and let $p:X\to Y$ be the quotient map
and endow $Y$ with the measure $p_*\nu$. Using a choice of an a.e defined  Borel
section $s:Y\to X$ of $p$ we may identify $X\cong Y\times K$
as Borel spaces, 
and endow $X$ with
the product measure:
\[ \tilde{\nu}=p_{*}\nu \times \lambda \in \Prob(Y\times K)\cong \Prob(X). \]
where $\lambda$ is the Haar measure  on $K$. This makes $X$ a Lebesgue $G$-space. Define $\alpha:G\times Y\to K$
by setting $\alpha(g,y)$ to be the unique element in $K$ for
which $g.s(y)=\alpha(g,y).s(g.y)$. Then it is easy to check that $\alpha$
is a cocycle and that $\rho:(y,k)\mapsto\sigma(y)k$ defines an
isomorphism between the Lebesgue $G$-spaces $Y\times_{\alpha}K$
and $X$.
Using Theorem~\ref{thm:minimal-cocycle},
we assume as we may that $\alpha$ is a minimal cocycle, whose image is denoted
by $K_{\alpha}\leq K$.
We get a Borel, $G$-invariant and $K$-equivariant map
$X \to K/K_\alpha$. By $\nu$-ergodicity, this map is essentially constant. 
Upon conjugating $\alpha$ by an element of $K$, we assume as we may that the essential image is the basic coset $eK_\alpha$.
It follows that $\nu$ is supported on the subset $Y\times_\alpha K_\alpha$ of $Y\times_\alpha K=X$.
We thus assume as we may that $K=K_\alpha$ and $\alpha$ is a minimal cocycle. 

Using Theorem~\ref{thm:minimal-cocycle} once again we see that
$\tilde{\nu}$ is $G$-ergodic.
Noticing that $\tilde{\nu}$ is $\mu$-stationary,
we get by Lemma~\ref{lem:support-of-stationary-is-invariant and ergodic-iff-extremal} that it is extremal as such.
Observing that $\tilde{\nu}=\int_K k.\nu d\lambda(k)$, we conclude that $\nu=\tilde{\nu}$. 
It follows that the disintegration map $Y\to \Prob(Y\times_\alpha K)$ is given by 
$y\mapsto \delta_y\times \lambda$,
which is $G$-equivariant.
Thus, indeed, $X$ is relatively measure preserving with respect to $Y=X/K$.
\end{proof}

\subsection{Furstenberg-Poisson boundaries}

The following result is the fundamental theorem of Furstenberg-Poisson theory,
see \cite[\S2]{bader2006factor}
for a detailed overview. In fact, it could be taken as the definition of Furstenberg-Poisson boundary.

\begin{thm}[Furstenberg] \label{thm:Poisson-boundary-universal-property}
Consider a locally compact second countable group $G$ and an admissible probability measure $\mu$
on $G$.
The Furstenberg-Poisson boundary of $(G,\mu)$ is a Lebesgue $G$-space $\Pi(G,\mu)$ endowed with
a $\mu$-stationary probability measure $\nu$
which satisfies the following property:
for every metrizable compact convex $G$-set $C$,
the correspondence that associates with a Lebesgue $G$-map $\theta:\Pi(G,\mu)\to C$
the barycenter $\bary{\theta_*\nu}\in C$ is a bijection between the collection of all such 
Lebesgue $G$-maps and the space $C^\mu$ of stationary points in $C$. 

Moreover, every Lebesgue $G$-space $\Omega$ endowed with a stationary measure $\eta$
which satisfies this property admits a unique morphism of Lebesgue $G$-spaces $\Omega \to \Pi(G,\mu)$
and this morphism pushes $\eta$ to $\nu$.
The above uniquely defines $\Pi(G,\mu)$ up to a unique isomorphism of Lebesgue $G$-spaces.
\end{thm}

An immediate consequence of the above result is the following useful criterion for invariance of a measure.

\begin{lem} \label{lem:invcrit}
With the notation above, a point $c\in C^\mu$ is $G$-fixed iff
the corresponding map $\Pi(G,\mu)\to C$ is essentially constant.
\end{lem}

We next discuss passage to finite index subgroups.

\begin{defn} \label{def:the-hitting-measure}
Consider a locally compact second countable group $G$ and an open subgroup of finite index $G'<G$.
Let $\mu$ be an admissible probability measure on $G$.
The associated \emph{hitting time} is the function 
\[ \tau:G^\N\to \N\cup\infty, \quad (g_1,g_2,\ldots) \mapsto \inf\{ n\in\N\mid g_{n}\cdot...\cdot g_{1}\in G'\}. \]
By the law of large numbers, this function is $\mu^\N$-a.e finite, thus the \emph{hitting map} 
\[ h:G^\N\to G', \quad w=(g_1,g_2,\ldots) \mapsto g_{\tau(w)}\cdot...\cdot g_{1} \]
is $\mu^\N$-a.e defined.
The \emph{hitting measure} on $G'$ is defined to be $\mu'=h_*(\mu^\N)$.
\end{defn}

The hitting measure is discussed in a general context in \cite{hartman2014abramov}.
In particular, the following is proved.

\begin{prop} \label{prop:hitting-measure-Poisson-boundary} 
Consider a locally compact second countable group $G$ and an open subgroup of finite index $G'<G$.
Let $\mu$ be an admissible probability measure on $G$ and let $\mu'$ be the corresponding hitting measure on $G'$.
Then the restriction to $G'$ of the Furstenberg-Poisson boundary of $(G,\mu)$ is (uniquely isomorphic to) 
the Furstenberg-Poisson boundary of $(G',\mu')$.
Moreover, for every metrizable compact convex $G$-set $C$, $C^{\mu}\subset C^{\mu'}$.
\end{prop}

\begin{proof}
The first part of the proposition is proved in \cite[Corollary 2.14]{hartman2014abramov}
under the assumption that $G$ is discrete, however an examination of the proof clearly reveals that the discreteness assumption 
is superfluous.
The second part of the proposition follows from the first by Theorem~\ref{thm:Poisson-boundary-universal-property}.
\end{proof}

\begin{lem} \label{lem:stiffness-subgroup}
Consider a locally compact second countable group $G$ and an open normal subgroup of finite index $G'\lhd G$.
Let $\mu$ be an admissible probability measure on $G$ and let $\mu'$ be the corresponding hitting measure on $G'$.
 If a Polish $G$-space $X$ is $\mu'$-stiff
then it is also  $\mu$-stiff.
\end{lem}

\begin{proof}
Suppose $X$ is $\mu'$-stiff, denote $C=\Prob(X)$ and let $c\in C^\mu$.
By Theorem~\ref{prop:hitting-measure-Poisson-boundary} above,
$c\in C^{\mu}\subseteq C^{\mu'}=C^{G'}$. Let $\bar{\mu}$
be the pushforward measure under the quotient map $G\to\bar{G}:=G/G'$.
The set $C^{G'}$ is $G$-invariant, and the action of $G$ on it factorizes through $\bar{G}$. It follows that by Lemma \ref{lem:stationarity-of-quotients} that $c\in\left(C^{G'}\right)^{\bar{\mu}}$.
But $\bar{G}$ is finite so by Lemma~\ref{lem:finite-is-stiff},
$c\in\left(C^{G'}\right)^{\bar{G}}=C^{G}$.
\end{proof}

\begin{lem}
\label{lem:If-kernel-acts-stiffly-on-fibers-then-also-Gamma} Consider a locally compact second countable group $G$ with an admissible measure $\mu$ on it. Let $p:X\to Y$
be a $G$-equivariant continuous and surjective map between Polish
$G$-spaces. Assume that $G':=\ker\left(G\curvearrowright Y\right)$
is an open subgroup of finite index in $G$ and let $\mu'$ denote the hitting measure on it. If every fiber of $p$ is $\mu'$-stiff
then $X$ is $\mu$-stiff.
\end{lem}

\begin{proof}
By Lemma \ref{lem:stiffness-subgroup} it suffices to show that
$X$ is $\mu'$-stiff. Given $\nu\in \Prob(X)$ we may disintegrate
it along $p$, that is write $\nu=\int_{Y}\nu_{y}\dd\left(p_{*}\nu\right)$
where $\nu_{y}\in \Prob(X)$ is defined for a.e $y\in Y$ and is supported
on the fiber $p^{-1}(y)$. Hence $\mu^{\prime}*\nu=\int_{Y}\mu'*\nu_{y}\dd(p_{*}\nu)$
and each $\mu'*\nu_{y}$ is again a probability measure on $X$ which
is supported on $p^{-1}(y)$. Therefore if $\nu$ is $\mu'$-stationary
then we have found two disintegration of $\nu$ along $p$, so by
uniqueness of the disintegration measures it must be that $\nu_{y}$
is $\mu'$-stationary for $p_{*}\nu$-a.e $y$. Since the action of
$G'$ on each fiber  is $\mu'$-stiff, we see that the measures $\nu_{y}$ are $G'$-invariant a.e. Thus
 $\nu$ is $G'$-invariant as desired. 
\end{proof}


The following result of Margulis is a generalization of \cite[Theorem 3]{furstenberg1967poisson}.

\begin{prop} \label{prop:hitlattice}
Let $G$ be a compactly generated locally compact group with left Haar measure $\lambda$, and let $\Gamma<G$ be a lattice. Let 
$\psi:G\to [0,\infty)$ be a continuous function with $\int_G \psi \dd\lambda=1$, which is  supported on a compact subset that generates $G$ as a semigroup, and consider the probability measure $\mu=\psi\cdot \lambda$ on $G$.
Then there exists a fully supported probability measure $\mu'$ on $\Gamma$ such that 
the restriction to $\Gamma$ of the Furstenberg-Poisson boundary of $(G,\mu)$ is
$\mu'$-stationary
and for every metrizable compact convex $G$-set $C$, $C^{\mu}\subset C^{\mu'}$.
Moreover, given any function  $\phi:\Gamma\to (0,1]$, we may choose $\mu'\leq \phi$.
\end{prop}

\begin{proof}
This is essentially the content of \cite[Proposition~VI.4.1]{margulis1991discrete}.
This reference deals with an $S$-arithmetic setting, but this is superfluous
and it serves as a complete proof that the $\mu$-stationary measure on $\Pi(G,\mu)$
is $\mu'$-stationary for some fully supported $\mu'\in\Prob(\Gamma)$.
Moreover, given any function $\phi:\Gamma\to (0,1]$
one can impose throughout the construction the condition $\mu'(\gamma)\leq \phi(\gamma)$.
The fact that $C^{\mu}\subset C^{\mu'}$ follows by Theorem~\ref{thm:Poisson-boundary-universal-property}.
\end{proof}

\subsection{Furstenberg measures} \label{sec:FM}

We have defined the notion of a ``Furstenberg measure'' in the introduction, see Definition~\ref{defn:Furstenberg-arithmetic}.
The following is a slight generalization of a classical result of Furstenberg, \cite[Theorem 3]{furstenberg1967poisson},
which regards the case in which $\Gamma$ is contained in the identity component of $\bG(\R)$.

\begin{prop} \label{prop:Fmeasures}
Let $\bG$ be an $\R$-algebraic group with a semisimple identity component and let $\Gamma<\bG(\R)$ be a lattice.
Then there exists a Furstenberg measure $\mu$ on $\Gamma$.
%
\end{prop}

\begin{proof}
Denote $G=\bG(\R)$ and let $G^o$ be its identity component.
Let $P<G^o$ be a minimal parabolic subgroup and $K<G^o$ be a maximal compact subgroup of this connected semisimple Lie group.
Since all minimal parabolic subgroups, as well as all maximal compact subgroups are conjugated and $P$ and $K$ are self normalizers in $G^o$,
we may identify the spaces $G^o/P$ and $G^o/K$ with the corresponding sets of 
minimal parabolic and maximal compact subgroups in $G^o$.
Thus, the conjugation action of $G$ on $G^o$ induces transitive actions on $G^o/P$ and $G^o/K$
correspondingly and we may identify $G^o/P\simeq G/N_G(P)$ and $G^o/K\simeq G/N_G(K)$.

Endow $G$ with the left Haar measure $\lambda$, and let 
$\psi:G\to [0,\infty)$ be a continuous function with $\int_G \psi \dd\lambda=1$ whose support is compact and generates $G$ as a semigroup.
and consider the probability measure $\bar{\mu}=\psi\cdot \lambda$ on $G$.
Find a $\bar{\mu}$-stationary probability measure $\nu$ on the compact $G$-space $G/N_G(P)$.
By Proposition~\ref{prop:hitlattice}, there exists a fully supported probability measure $\mu$ on $\Gamma$ such that 
$\nu$ is $\mu$-stationary.
We choose $\bar{\mu} \leq \phi$ for $\phi:\Gamma\to (0,1]$ small enough so that $\mu$ has finite logarithmic first moment 
and random walk entropy.
We now use \cite[Theorem~10.7]{Kaimanovich2000Poisson} to deduce that $(G/N_G(P),\nu)$ is the 
Furstenberg-Poisson boundary of $(\Gamma,\mu)$.
Strictly speaking, it is assumed in this reference that $G$ is connected, but since $G$ acts by isometries on the symmetric space
$G^o/K$, the use of the strip approximation \cite[Theorem~6.5]{Kaimanovich2000Poisson} is intact.
Since $N_G(P)=\bN(\R)$ where $\bN$ is the normalizer of the Zariski closure of $P$, 
we get that $G/N_G(P)=\bG(\R)/\bN(\R)$,
as required to show that $\mu$ is a Furstenberg measure on $\Gamma$.
%
\end{proof}

\begin{proof}[Proof of Proposition~\ref{prop:furstmeasures-arithmetic}]
Let $\bG$ be a $\Q$-algebraic group and $\Lambda$ an arithmetic subgroup.
Let $\bR\lhd \bG$ be the solvable radical of the identity component subgroup in $\bG$ and denote $\bar{\bG}=\bG/\bR$.
Let $\bar{\Lambda}$ be the image of $\Lambda$ in $\bar{\bG}(\Q)$.
By \cite[Proposition 4.1]{platonov-rapinchik-1993} $\bar{\Lambda}$ is an arithmetic subgroup of $\bar{\bG}$,
thus by a classical theorem of Borel and Harish-Chandra it is a lattice in $\bar{\bG}(\R)$.
Using Proposition~\ref{prop:Fmeasures}, we find a Furstenberg measure $\bar{\mu}$ on $\bar{\Lambda}$.
Then the Furstenberg-Poisson boundary of $(\bar{\Lambda},\bar{\mu})$ is realized by $\bar{\bG}(\R)/\bar{\bN}(\R)$,
where $\bar{\bN}<\bar{\bG}$ is the normalizer of some minimal $\R$-parabolic subgroup $\bar{\bP}$ of $\bar{\bG}$.
We denote by $\bP$ and $\bN$ the preimages of $\bar{\bP}$ and $\bar{\bN}$ in $\bG$ correspondingly 
and observe that $\bP$ is a minimal $\R$-parabolic subgroup and 
$\bN$ is its normalizer in $\bG$.
Since the kernel of $\Lambda\to \bar{\Lambda}$ is solvable, hence amenable,
we can use \cite[Theorem~2]{kaimanovich2002extensions} to find a probability measure $\mu$ on $\Lambda$
such that the Furstenberg-Poisson boundary of $(\Lambda,\mu)$ is the same as the Furstenberg-Poisson boundary of 
$(\bar{\Lambda},\bar{\mu})$, that is $\bar{\bG}(\R)/\bar{\bN}(\R)$.
Identifying the latter with $\bG(\R)/\bN(\R)$,
we conclude that $\mu$ is a Furstenberg measure on $\Lambda$.
\end{proof}

\begin{rem}
	We note that the same proof applies for S-arithmetic subgroups of $k$-algebraic groups where $k$ is a number field and $S$ is a finite set of places. 
\end{rem}

\section{Characters} \label{sec:Characters}

Throughout this section we fix a countable group $\Lambda$. All representations
are assumed to be unitary. For a comprehensive overview of the theory
of characters and more generally of dual spaces in harmonic analysis we refer to
\cite{bekka2019unitary}.

\subsection{Basics}  \label{subsec:Basics}

A function $\varphi:\Lambda\to\C$ is called \emph{positive-definite} if
for any $n\in\N$, and for any group elements $g_{1},...,g_{n}\in \Lambda$,
the matrix $\left(\varphi(g_{i}^{-1}g_{j})\right)_{i,j}$ is positive
semi-definite, that is, $\sum_{i,j}\bar{\alpha_{i}}\alpha_{j}\varphi(g_{i}^{-1}g_{j})\geq0$
for any choice of complex numbers $\alpha_{1},...,\alpha_{n}$.
The space of such functions is denoted by $\PD(\Lambda)$. It is not hard
to see that $\varphi(e)=\sup_{g\in \Lambda}|\varphi(g)|$ and in particular
$\PD(\Lambda)\subseteq l^{\infty}(\Lambda)$. We may therefore speak of the weak-{*}
topology on $\PD(\Lambda)$, which makes $\PD(\Lambda)$ a close convex set.
The space $\PD_{1}(\Lambda)$ obtained by adding the normalization condition $\varphi(e)=1$
is a then compact and convex. The weak-{*} topology
on $\PD_{1}(\Lambda)$ coincides with the topology of pointwise convergence. The restriction to $\Lambda$ of any state on   the (maximal) $C^*$-algebra $C^*(\Lambda)$ is positive-definite and conversely any $\varphi\in \PD_1(\Lambda)$ can be extended to $C^*(\Lambda)$.
This establishes a bijection between $\PD_{1}(\Lambda)$ and the space of states on $C^*(\Lambda)$.

An element $\varphi\in \PD_{1}(\Lambda)$
which is conjugation invariant is called a \emph{trace}. Traces on $\Lambda$ correspond to tracial states on $C^*(\Lambda)$. The collection of all traces $\Tr(\Lambda)$ is again compact and
convex, so it is natural to speak of extreme points in $\Tr(\Lambda)$,
that is traces that cannot be written as a proper convex sum of two
traces. Any extremal point of this set is called a \emph{character
}of $\Lambda$, and the collection of all characters is denoted by $\Ch(\Lambda)$\footnote{Note that in some texts, e.g \cite{bader2022charmenability},
	the term character refers to what we call here a trace.}.
Observe that $\Tr(\Lambda)$ is metrizable, as $\Lambda$ is countable, thus its $G_\delta$-subset $\Ch(\Lambda)$ is Polish 
(see \cite[Proposition 1.3]{phelps2001lectures}), but it is in general not compact, see e.g Example~\ref{exa:heisenberg}.
The space $\Ch(\Lambda)$ is sometimes referred to as the \emph{Thoma dual} of $\Lambda$ as it serves as a dual space of
$\Lambda$ for harmonic analytic purposes. 
The following theorem may be thought of as a Fourier transform on $\Lambda$.

\begin{thm}[\cite{thoma1964unitare}] \label{thm:fourier}
	$\Tr(\Lambda)$ is a Choquet simplex. That is, the barycenter map
	\[ \mathcal{F}_{\Lambda}:\Prob(\Ch(\Lambda))\to \Tr(\Lambda), \quad
	\nu\mapsto\int_{\Ch(\Lambda)}\chi \dd\nu
	\]
	is a continuous affine bijection. 
\end{thm}

\begin{proof}[Proof of Proposition \ref{Char-kakutani}]
	The Proposition follows from Lemma \ref{lem:choquet-kakutani}	 in view of Theorem~\ref{thm:fourier} and the metrizability of $\Tr(\Lambda)$.
\end{proof}

\subsection{The GNS construction}
A \emph{GNS triplet} of $\Lambda$ is a triple $(\mathcal{H},\pi,\xi)$ where $\mathcal{H}$ is a Hilbert space, $\pi:\Lambda\to \mathcal{U}(\mathcal{H})$ is a representation, and $\xi$ is a cyclic vector. Given such, the function
\begin{align} \label{eq:pd}
	\varphi:g\mapsto \langle \pi(g)\xi,\xi\rangle, \;\;\;\;\; (g\in \Lambda)
\end{align} 
is in $\PD({\Lambda})$. Conversely, the GNS construction associates to any $\varphi\in \PD(\Lambda)$ a GNS triplet $(\mathcal{H},\pi,\xi)$ such that Eq. (\ref{eq:pd}) holds, and this triplet is unique up to an equivalence of representations that takes the respective cyclic vectors to one another \cite[Proposition 1.B.10]{bekka2019unitary}. Furthermore, $\varphi\in \PD_1(\Lambda)$ if and only if $\xi$ is a unit vector, and $\varphi$ is extremal if and only if $\pi$ is irreducible \cite[Proposition 1.B.14]{bekka2019unitary}.

\begin{lem}\label{lem:finite-decomposition-rep-pd}
	Let $\varphi\in \PD_1(\Lambda)$ with GNS triplet $(\mathcal{H},\pi,\xi)$ and let $d\in \N$. If
	$\pi$ admits a decomposition into $d$ irreducible subrepresentations then $\varphi$ is a convex combination of $d$ extreme points of $\PD_1(\Lambda)$.
\end{lem} 

\begin{proof}
	
	Suppose $\pi=\bigoplus_{i=1}^d \pi_i$ is a decomposition into $d$ irreducible subrepresentations. For each $1\leq i \leq d$ consider the GNS triplet $(\mathcal{H}_i,\pi_i,\xi_i)$ where $\mathcal{H}_i$ is the space of $\pi_i$, and $\xi_i$ is the projection of $\xi$ onto $\mathcal{H}_i$. Denote the corresponding positive-definite functions to these triplets by $\varphi_i$ and note that $\varphi=\sum_{i=1}^{d} \varphi_i $. In particular, $\varphi$ is a convex combination of the normalized positive-definite functions $\frac{1}{\varphi_i(1)}\varphi_i$ (with respective coefficients $\varphi_i(1)$). Since each $\pi_i$ are irreducible, each $\frac{1}{\varphi_i(1)}\varphi_i$ is extremal.	 
\end{proof}

\begin{defn}
	Let $\varphi\in \PD_1(\Lambda)$. The \emph{support} of $\varphi$
	is $\supp(\varphi)=\Lambda\backslash\varphi^{-1}(0)$. 
	The \emph{kernel} of $\varphi$ is
	$\ker\varphi:=\varphi^{-1}(1)$.  $\varphi$ is called
	\emph{faithful} if $\ker\varphi=\{e\}$
\end{defn}

\begin{lem}\label{lem:kernel-of-pd}
	For any $\varphi\in\PD_1(\Lambda)$, $\ker\varphi$ is a subgroup of
	$\Lambda$, and $\varphi(k_{1}gk_{2})=\varphi(g)$ for all $k_{1},k_{2}\in\ker\varphi$
	and $g\in\Lambda$. In particular, if $\varphi$ is a trace then $\ker\varphi$
	is a normal subgroup and $\varphi$ descends to a faithful trace of
	$\Lambda/\ker\varphi$. 
\end{lem}

\begin{proof}
	Let $\left(\mathcal{H},\pi,\xi\right)$ be a GNS triplet such that
	$\varphi(g)=\left\langle \pi(g)\xi,\xi\right\rangle $. Then $\ker\varphi=\left\{ g\in\Lambda\mid\pi(g)\xi=\xi\right\} $
	and in particular it is a subgroup. Given $k_{1},k_{2}\in\ker\varphi$ we see that:
	\[
	\varphi(k_{1}gk_{2})=\left\langle \pi(k_{1}gk_{2})\xi,\xi\right\rangle =\left\langle \pi(g)\pi(k_{2})\xi,\pi(k_{1}^{-1})\xi\right\rangle =\left\langle \pi(g)\xi,\xi\right\rangle =\varphi(g)
	\]
	Now if $\varphi$ is a trace, $\ker\varphi$ is a normal subgroup, and by the above it is constant on $\ker\varphi$-cosets. It therefore descends to a well defined function $\bar{\varphi}$ of
	$\Lambda/\ker\varphi$. It is clear that $\bar{\varphi}$ is a trace and that it is faithful. 
\end{proof}

\begin{defn}
	Let $(\mathcal{H},\pi,\xi)$ be a GNS triplet of a subgroup $\Gamma\leq \Lambda$. The \emph{induced GNS triplet} (from $\Gamma$ to $\Lambda$) is $(\tilde{\mathcal{H}},\tilde{\pi},\tilde{\xi})$ where $\tilde{\mathcal{H}}=l^2(\Lambda/\Gamma)\otimes \mathcal{H}$ is the Hilbertian tensor product, $\tilde{\pi}=\Ind_{\Gamma}^{\Lambda}\pi$ is the induced representation, and $\tilde{\xi}=\delta_{e\Gamma}\otimes \xi$
\end{defn}

Let $\Gamma$ be a subgroup of $\Lambda$. Given $\varphi\in \PD_1(\Gamma)$ we define its \emph{trivial extension} $\tilde{\varphi}:\Lambda\to \C$ by:
\[
\tilde{\varphi}(g)=\begin{cases}
	\varphi(g) & g\in \Gamma\\
	0 & g\notin \Gamma
\end{cases}
\]
The fact that $\tilde{\varphi}$ is positive-definite follows by the following Lemma:

\begin{lem} 
	\label{lem:induction-trivial-extension} Let $\Gamma\leq \Lambda$, and let $\varphi\in \PD_{1}(\Gamma)$ with GNS triplet $(\mathcal{H},\pi,\xi)$. Then the positive-definite function corresponding to the induced GNS triplet $(\tilde{\mathcal{H}},\tilde{\pi},\tilde{\xi})$ is	$\tilde{\varphi}$.
\end{lem}

\begin{proof}
	By definition of the induced representation, each $\tilde{\pi}(g)$ permutes the pairwise orthogonal subspaces $\delta_{x\Gamma}\otimes \mathcal{H}$ ($x\Gamma \in \Lambda/\Gamma$),  and stabilizes the subspace $\delta_{e\Gamma}\otimes \mathcal{H}$ if and only if $g\in \Gamma$. Moreover, the subrepresentation of the restricted representation $\Res_{\Gamma}^\Lambda \tilde{\pi}$ corresponding to the $\Gamma$-invariant subspace $\delta_{e\Gamma}\otimes \mathcal{H}$ is simply $\pi$. It follows that:
	\[
		\langle \tilde{\pi}(g)\tilde{\xi},\tilde{\xi} \rangle = 
		\langle \Res^{\Lambda}_{\Gamma}\tilde{\pi}(g)\tilde{\xi},\tilde{\xi} \rangle 
		=
			\langle \pi(g)\xi,\xi \rangle
		=\varphi(g)
	\]
	for $g\in \Gamma$, whereas $\langle \tilde{\pi}(g)\tilde{\xi},\tilde{\xi} \rangle = 0 $ for 
$g\notin\Gamma$. Altogether $\tilde{\varphi}(g)=\langle \tilde{\pi}(g)\tilde{\xi},\tilde{\xi} \rangle$ for all $g\in \Lambda$ 
\end{proof}

\begin{lem}\label{lem:decompose-trivial-extension-pd}
	Let $N$ be a normal subgroup of index $d<\infty$ in $\Lambda$, and let
	$\varphi$ be an extreme point of $\PD_1(N)$. Then the trivial extension
	$\tilde{\varphi}\in\PD_1({\Lambda})$ is a convex combination of $d$ (or less) extreme points of $\PD_1({\Lambda})$.
\end{lem}

\begin{proof}
	Let $(\mathcal{H},\pi,\xi)$ be a GNS triplet associated with $\varphi$.
	Then by Lemma \ref{lem:induction-trivial-extension}, the induced GNS triplet $(\tilde{\mathcal{H}},\tilde{\pi},\tilde{\xi})$ is a GNS triplet for $\tilde{\varphi}$. As $N$ is normal $\Res^{\Lambda}_N\Ind^{\Lambda}_N\pi=\bigoplus_{i=1}^{d}\pi^{t_{i}}$ where $t_1,...,t_d$ are coset representatives for $N$ in $\Lambda$, and $\pi^{t_i}$ is the representation of $N$ defined by $n\mapsto \pi(t_i^{-1}nt_i)$.
	Since $\varphi$ is extremal, $\pi$
	is irreducible and thus also each $\pi^{t_{i}}$. It follows that $\Ind^{\Lambda}_N\pi$ is a direct sum
	of at most $d$ irreducible representations because this is true even when
	restricting to $N$. Thus by Lemma \ref{lem:finite-decomposition-rep-pd} $\tilde{\varphi}$ is a convex combination
	of $d$ extreme points of $\PD_1({\Lambda})$. 
\end{proof}

\subsection{Tracial representations  \label{subsec:Tracial-representations}}

A \emph{tracial von Neumann algebra} is a pair $(M,\tau)$,
where $M$ is a von Neumann algebra and $\tau$ is a normal faithful trace on $M$. A \emph{finite factor} is a tracial von Neumann algebra whose center consists only of scalar multiples of the identity. 
A \emph{tracial representation} of a group $\Lambda$ is a triple $(M,\tau,\pi)$ where $(M,\tau)$ is a tracial von Neumann algebra and 
$\pi:\Lambda\to \mathcal{U}(\mathcal{M})$ is a representation of $\Lambda$ into the group of
unitary elements of $M$ such that $\pi(\Lambda)$ generates $M$ as a von Neumann algebra. Two tracial representations are said to be equivalent if there is an isomorphism between the two von Neumann algebras which intertwines the traces and the representations. 
The following is a fundamental result in the theory of characters.

\begin{thm}[Thoma]	\label{thm:thoma correspondence}   
	There is a one to one correspondence	between:
	\begin{enumerate}
		\item Traces $\varphi:\Lambda\to\C$, and 
		\item Isomorphism classes of tracial representations of $\Lambda$.
	\end{enumerate}
	Moreover $\varphi$ is a character if and only if the corresponding
	von Neumann algebra is a factor. 
\end{thm}

For the proof see \cite[Lemmas 11.C.3 and 4]{bekka2019unitary},
which are formulated in a different, yet equivalent language, in view of \cite[Proposition 6.B.5]{bekka2019unitary}.
Briefly, given a tracial representation $(M,\tau,\pi)$,
$\phi=\tau\circ\pi$ is a trace on $\Lambda$ and we can retrieve 
$(M,\tau,\pi)$ from $\phi$ by the GNS construction. 

\begin{cor}\label{cor:abelian-thoma=pontryagin}
	For an abelian group $\Lambda$, the space $\Ch(\Lambda)$ coincides with the Pontryagin dual $\hat{\Lambda}$. 
\end{cor}
\begin{proof}
	Given $\varphi\in \Ch(\Lambda)$ let $(M,\tau,\pi)$ be its associated tracial representation. Then $M$ is a factor, and it is commutative, thus isomorphic to $\C$. It follows that $\pi$ is a homomorphism into the unit circle. 
\end{proof}

\begin{defn} \label{def:vNamenable}
	A von Neumann algebra is called \emph{amenable} if for some (equivalently, any) embedding $M\subseteq \mathcal{B}(\mathcal{H})$, there exists a conditional expectation $\mathcal{B}(\mathcal{H})\to M$. A trace $\varphi\in \Tr(\Lambda)$, with corresponding tracial representation $(M,\tau,\pi)$, is called \emph{von Neumann amenable} if $M$ is amenable.
\end{defn}

Any finite dimensional von Neumann algebra is amenable. It is well known that all separable infinite-dimensional amenable  finite factors are isomorphic to a single von Neumann algebra $\mathcal{R}$, often referred to as the hyperfinite  $\twoone$-factor \cite{murray1936rings,connes1976classification}. We refer to \cite[Chapter 11]{anantharaman2017introduction} for more on this manner.

\subsection{Relative characters and stiffness}\label{subsubsec:relative-characters}

Traces and characters generalize to a relative setting. Let $\Gamma$ be a group acting on $\Lambda$ by group automorphisms.
Then $\Gamma$ acts on $\PD_1(\Lambda)$  by $\varphi\mapsto \varphi^\gamma$, where $\varphi^\gamma(g)=\varphi(\gamma^{-1}(g))$, and $\Tr(\Lambda)$ is an invariant subset. 
We define the space of \emph{$\Gamma$-traces} of $\Lambda$, $\Tr_{\Gamma}(\Lambda)$, to be the $\Gamma$-fixed points of $\Tr(\Lambda)$
and the space of \emph{$\Gamma$-characters} of $\Lambda$, $\Ch_\Gamma(\Lambda)$, to be extreme points of $\Tr_\Gamma(\Lambda)$.

\begin{thm} \label{thm:fourierinvariants}
The compact convex space $\Tr_{\Gamma}(\Lambda)$ is a Choquet simplex,
that is 
\[ \Prob(\Ch_\Gamma(\Lambda)) \to \Tr_{\Gamma}(\Lambda) \]
is a continuous affine bijection.
%
\end{thm}

\begin{proof}
By Theorem~\ref{thm:fourier}
$\Tr(\Lambda)$ is a Choquet simplex, thus by Proposition~\ref{prop:choqueinv} so is 
its subspace of $\Gamma$-fixed points.
\end{proof}

%
%

We now restrict to the case where $N$ is a normal subgroup of $\Lambda$ so that $\Lambda$ acts on it by conjugation. 
For any $\varphi\in \Tr(\Lambda)$ the restriction
$\varphi\res_{N}$ is easily seen to be in $\Tr_{\Lambda}(N)$. We call the mapping $r:\varphi\to\varphi\res_{N}$
the \emph{restriction map}. Obviously it is continuous, and it
surjects onto $\Tr_{\Lambda}(N)$. Indeed, the trivial extension $\tilde{\varphi}$ is a trace of $\Lambda$ whose restriction to $N$ is $\varphi$.

\begin{lem}
	\label{lem:restriction-of-characters}The restriction map $r:\Tr(\Lambda)\to \Tr_{\Lambda}(N)$
	maps $\Ch(\Lambda)$ surjectively onto $\Ch_{\Lambda}(N)$. In particular if $N$
	is in the center of $\Lambda$, then $r$ maps $\Ch(\Lambda)$ onto the Pontryagin dual
	$\hat{N}$ of $N$.
\end{lem}

\begin{proof}
	The fact that $r$ takes $\Ch(\Lambda)$ to $\Ch_{\Lambda}(N)$ is shown in \cite[Lemma 14]{thoma1964unitare}.
	The fact that $r$ is surjective is shown in \cite[Lemma 16]{thoma1964unitare}.
	Finally, if $N$ is central then $\Lambda$ acts trivially on $\Tr(N)$ and
	therefore $\Ch_{\Lambda}(N)$ is equal to $\Ch(N)$ which is simply $\hat{N}$
	by Corollary \ref{cor:abelian-thoma=pontryagin}.
\end{proof}


\begin{prop} \label{prop:finite-index-finite-fibers}
	Suppose $N$ is a normal subgroup of finite index in $\Lambda$.  
	Then each fiber of the restriction map $\Ch(\Lambda)\to \Ch_{\Lambda}(N)$ is of size at most $[\Lambda:N]$. 
\end{prop}

The proof of Proposition~\ref{prop:finite-index-finite-fibers}
will be given in \S\ref{subsec:pd-vs-tr}. The following is an important conclusion regarding stiffness of actions on spaces of characters, and will be used often. 

\begin{prop}
	\label{prop:Ch(N) stiff then also Ch_H(N)}
	Fix a group $\Gamma$ and an admissible measure $\mu$ on $\Gamma$.
	Assume that $\Gamma$ acts by automorphisms
	on a group $\Lambda$ and let $N$ be a $\Gamma$-invariant normal subgroup
	of $\Lambda$. 
	Consider the following statements regarding the corresponding $\Gamma$-actions on $\Ch(\Lambda)$,
	$\Ch(N)$ and $\Ch_{\Lambda}(N)$.
	\begin{enumerate}
		\item $\Gamma\curvearrowright \Ch(N)$ is stiff.
		\item $\Gamma\curvearrowright \Ch_{\Lambda}(N)$ is stiff.
		\item $\Gamma\curvearrowright \Ch(\Lambda)$ is stiff.
	\end{enumerate}
	Then 1 implies 2,  3 implies 2 if the fibers of $\Ch(\Lambda)\to \Ch_\Lambda(N)$ are compact, and all three
	are equivalent if  $[\Lambda:N]<\infty$. 
\end{prop}

\begin{proof}
	By Theorem~\ref{thm:fourierinvariants}, $\Tr(N)$ and $\Tr_\Lambda(N)$ are Choquet simplices thus 
	we have $\Gamma$-equivariant bijections 
	\begin{equation} \label{eq:relchar}
		\Prob(\Ch_{\Lambda}(N)) \simeq \Tr_{\Lambda}(N) =  \Tr(N)^{\Lambda}  \simeq  \Prob(\Ch(N))^{\Lambda}.
	\end{equation}
	Note, however that these are in general not homeomorphisms.
Nevertheless, the stiffness of all these spaces are equivalent by Lemma~\ref{lem:choquet-kakutani}.
	
	Stiffness of the Polish space $\Ch(N)$ is by definition equivalent to stiffness of the convex space $\Prob(\Ch(N))$
	which implies the stiffness of its $\Gamma$-invariant convex subspace $\Prob(\Ch(N))^\Lambda$,
	hence by Equation~\eqref{eq:relchar}, of $\Prob(\Ch_{\Lambda}(N))$ which is equivalent (by definition) to the stiffness of the Polish space $\Ch_{\Lambda}(N)$.
	This shows that $1$ implies $2$. 
	
	By Lemma~\ref{lem:restriction-of-characters}, the map $\Ch(\Lambda)\to \Ch_\Lambda(N)$ is surjective, thus by 
	Lemma~\ref{lem:factor-of-stiff-is-stiff}, if its fibers are compact then $3$ implies $2$.
	
	Assume that $[\Lambda:N]<\infty$. 
	Then $2$ implies $3$ by Proposition
	\ref{prop:finite-index-finite-fibers} and Proposition~\ref{prop:countable-fibers-stiffness}.
	Finally to see that $2$ implies $1$ note that Equation~\eqref{eq:relchar}
	identifies $\Ch_{\Lambda}(N)$ with ergodic $\Lambda$-invariant probability measures
	on $\Ch(N)$. All such measures are uniform measures on finite orbits
	thus $\Ch_{\Lambda}(N)$ is identified
	with the space of orbits $\Ch(N)/\Lambda$. The statement then follows from
	Proposition~\ref{prop:countable-fibers-stiffness}. 
\end{proof}

\begin{cor} 	\label{cor:stiff commen}
	Fix a group $\Gamma$ and an admissible measure $\mu$ on $\Gamma$.
	Assume that $\Gamma$ acts by automorphisms
	on a group $\Lambda$.
If $\Lambda$ is finitely generated and 
$\Lambda_0<\Lambda$ is a $\Gamma$-invariant subgroup of finite index
then 
$\Gamma\curvearrowright \Ch(\Lambda)$ is stiff iff
$\Gamma\curvearrowright \Ch(\Lambda_0)$ is stiff.
If, further, $\Lambda$ is residually finite and 
$F\lhd \Lambda$ is a $\Gamma$-invariant finite normal subgroup
then 
$\Gamma\curvearrowright \Ch(\Lambda)$ is stiff iff
$\Gamma\curvearrowright \Ch(\Lambda/F)$ is stiff.
\end{cor}

\begin{proof}
Assume $\Lambda$ is finitely generated and 
let $\Lambda_0<\Lambda$ be a $\Gamma$-invariant subgroup.
Using the fact that $\Lambda$ is finitely generated, 
it has only finitely many subgroups of index $[\Lambda:\Lambda_0]$,
thus their intersection $N$ is a characteristic normal subgroup of finite index in $\Lambda$ which is contained in $\Lambda_0$.
In particular, $N$ is $\Gamma$-invariant.
By $1\Leftrightarrow 3$ in Proposition~\ref{prop:Ch(N) stiff then also Ch_H(N)}
we get that 
$\Gamma\curvearrowright \Ch(\Lambda)$ is stiff iff
$\Gamma\curvearrowright \Ch(N)$ is stiff iff
$\Gamma\curvearrowright \Ch(\Lambda_0)$ is stiff.

Assume further that $\Lambda$ is residually finite
and let $F\lhd \Lambda$ be a $\Gamma$-invariant finite normal subgroup.
Using the fact that $\Lambda$ is residually finite,
we find a normal subgroup of finite index $N\lhd \Lambda$ such that $F$ injects into $\Lambda/N$.
Arguing as above, we assume as we may that $N$ is characteristic, thus $\Gamma$-invariant.
Therefore, the $\Gamma$-action descents to an action on the finite group $\Lambda/N$.
We consider the $\Gamma$-equivariant injection $\Lambda \to \Lambda/F \times \Lambda/N$ 
and identify $\Lambda$ with its image, which is of finite index.
Applying twice the first part of the corollary, 
we get that 
$\Gamma\curvearrowright \Ch(\Lambda)$ is stiff iff
$\Gamma\curvearrowright \Ch(\Lambda/F \times \Lambda/N)$ is stiff iff
$\Gamma\curvearrowright \Ch(\Lambda/F \times \{e\})$ is stiff.
We conclude that $\Gamma\curvearrowright \Ch(\Lambda)$ is stiff iff $\Gamma\curvearrowright \Ch(\Lambda/F)$ is stiff.
\end{proof}

\subsection{Relative traces as positive-definite functions with kernels}\label{subsec:pd-vs-tr}
If $\varphi\in \Tr(\Lambda)$ is an extreme point of $\PD_1({\Lambda})$, then it is also an extreme point of $\Tr({\Lambda})$, but the converse
does not hold in general. To this matter we introduce a perspective
used to bridge over this gap, and in turn to complete the proof of Proposition \ref{prop:finite-index-finite-fibers}. We start with the following general statement:
\begin{prop} \label{prop:traces-vs-pd}
	Consider a group $\Gamma$ acting on $\Lambda$ by group automorphisms
and the associated semidirect product $\Gamma\ltimes\Lambda$.
Denote by $p:\Gamma\ltimes\Lambda\to\Lambda$ the (non-homomorphic) map $(\gamma,g)\mapsto g$.
	Then $\varphi\mapsto\varphi\circ p$ defines an injective affine
	continuous map $\PD_1({\Lambda})^{\Gamma}\to \PD_1({\Gamma\ltimes\Lambda})$. The image is the set of all $\psi\in \PD_1({\Gamma\ltimes\Lambda})$ with $\Gamma\subseteq \ker\psi$, and in particular it is a face of $\PD_1({\Gamma\ltimes\Lambda})$.
\end{prop}

\begin{proof}
	Let $\varphi\in\PD_1({\Lambda})^{\Gamma}$. Given $n\in\N$, $\gamma_{1},...,\gamma_{n}\in\Gamma$,
	$g_{1},...,g_{n}\in\Lambda$ and $\alpha_{1},...,\alpha_{n}\in \C$ write:
	\begin{align*}
	\sum_{i,j}\overline{\alpha_{i}}\alpha_{j}\left(\varphi\circ p\right)\left((\gamma_{i},g_{i})^{-1}\cdot(\gamma_{j},g_{j})\right)=
	\sum_{i,j}\overline{\alpha_{i}}\alpha_{j}\varphi\left(\gamma_{i}^{-1}\left(g_{i}^{-1}g_{j}\right)\right)=\sum_{i,j}\overline{\alpha_{i}}\alpha_{j}\varphi\left(g_{i}^{-1}g_{j}\right)\geq0
	\end{align*}
	Hence $\varphi\circ p \in \PD_1(\Gamma \ltimes \Lambda)$ (a more conceptual argument
	is obtained via the associated GNS representation).
	The fact that the mapping $\varphi\mapsto\varphi\circ p$ is injective,
	affine and continuous is straightforward. 
	
	Let $F$ denote the set of all $\psi\in\PD_1({\Gamma\ltimes\Lambda})$
	with $\Gamma\subseteq \ker\psi$. As $\psi$ attains values in the unit disk
	in $\C$, it follows that $F$ is a face of $\PD_1({\Gamma\ltimes\Lambda})$. Moreover, $\varphi\circ p\in F$ for any $\varphi\in \PD_1(\Lambda)^\Gamma$, so it is left to show surjectivity onto $F$. If $\psi\in \PD_1(\Gamma\ltimes \Lambda)$ then clearly $\psi\res_{\Lambda}\in\PD_1({\Lambda})$. Assuming moreover that $\psi\in F$ we deduce by Lemma \ref{lem:kernel-of-pd} that for all $\gamma\in\Gamma$ and $g\in\Lambda$:
	\[
	\psi\res_{\Lambda}(\gamma(g))=\psi\left(e_{\Gamma},\gamma(g)\right)=\psi\left((\gamma,e_{\Lambda})(e_{\Gamma},g)(\gamma^{-1},e_\Lambda)\right)=\psi(e_{\Gamma},g)=\psi\res_{\Lambda}(g)
	\]
 showing that $\psi\res_\Lambda\in \PD_1(\Lambda)^\Gamma$, and
	\[	\psi(\gamma,g)=\psi((e_\Gamma,g)(\gamma,e_\Lambda))=\psi(e_\Gamma,g)=\psi\res_\Lambda \circ p(\gamma,g) 
	\]
 showing that  $\psi=\psi\res_{\Lambda}\circ p$. 
\end{proof}

\begin{example}\label{example:traces-pd}
Fix a group $\Lambda$ and let $\Gamma=\Lambda$ act on it by inner automorphisms. 
We get the identification $\PD_1({\Lambda})^{\Gamma}=\Tr({\Lambda})$.
Note that we may also identify in this case $\Gamma\ltimes\Lambda\simeq \Lambda\times \Lambda$,
by $(\gamma,\lambda) \mapsto (\gamma,\lambda\gamma)$.
Under these identifications, the map 
$\PD_1({\Lambda})^{\Gamma}\to \PD_1({\Gamma\ltimes\Lambda})$
considered in Proposition~\ref{prop:traces-vs-pd},
identifies $\Tr({\Lambda})$ with the subset of elements of $\PD_1(\Lambda\times \Lambda)$
which are trivial on the diagonal subgroup.
This is a well known identification.
\end{example}  

The following is a generalization of the previous example.

\begin{example}
Consider a group $\Gamma$ acting on $\Lambda$ by group automorphisms
and assume that the image of $\Gamma$ in $\Aut(\Lambda)$ contains all inner automorphisms.
Then $\PD_1({\Lambda})^{\Gamma}=\Tr_{\Gamma}(\Lambda)$, so that $\Tr_\Gamma(\Lambda)$ embeds as a face of $\PD_1(\Gamma\ltimes \Lambda)$. In particular $\varphi\in \Tr_\Gamma(\Lambda)$ belongs to $\Ch_\Gamma(\Lambda)$ if and only if $\varphi\circ p$ is an extreme point of $\PD_1(\Gamma\ltimes \Lambda)$. 
\end{example}  

We will use the last example in the course of the proof of the following lemma.

\begin{lem}\label{lem:decompose-trivial-extension-tr}
	Let $N$ be a normal subgroup of index $d<\infty$ in $\Lambda$, and let
	$\varphi\in \Ch_{\Lambda}(N)$. Then the trivial extension $\tilde{\varphi}\in\Tr({\Lambda})$
	decomposes as a convex combination of $d$ (or less) characters of $\Lambda$. 
\end{lem}

\begin{proof}
	Consider the inner action of $\Lambda$ on itself, as well as on the
	invariant subset $N$. With respect to these actions form the semidirect
	products $\Lambda\ltimes\Lambda$ and $\Lambda\ltimes N$ and let
	$p_{\Lambda}:\Lambda\ltimes\Lambda\to\Lambda$ and $p_{N}:\Lambda\ltimes N\to N$
	denote the (set-theoretic) projections onto the second coordinates. By Proposition \ref{prop:traces-vs-pd}, showing that 
	$\tilde{\varphi}$ is a convex combination
	of $d$ characters is the same as showing that $\tilde{\varphi}\circ p_{\Lambda}$
	is a convex combination of $d$ extreme points of $\PD_1({\Lambda\ltimes\Lambda})$. 
	
	Now $\varphi$ is an extreme point of $\Tr_{\Lambda}(N)$ and therefore
	by Proposition \ref{prop:traces-vs-pd} $\varphi\circ p_{N}$ is an extreme point of $\PD_1({\Lambda\ltimes N})$.
	Hence, by Lemma \ref{lem:decompose-trivial-extension-pd} $\widetilde{\varphi\circ p_{N}}$, the trivial extension of $\varphi\circ p_N$ from $\Lambda\ltimes N$ to $\Lambda\ltimes\Lambda$, is a convex
	combination of $d$ extreme points of $\PD_1({\Lambda\ltimes\Lambda})$. But
	clearly $\tilde{\varphi}\circ p_{\Lambda}=\widetilde{\varphi\circ p_{N}}$,
	which completes the proof. 
\end{proof}

\begin{proof}[Proof of Proposition \ref{prop:finite-index-finite-fibers}]
	Let $\varphi \in \Ch_{\Lambda}(N)$ and let $\tilde{\varphi}$ be its trivial extension.
	By Lemma \ref{lem:decompose-trivial-extension-tr}, $\tilde{\varphi}$
	is a convex combination of $[\Lambda:N]$ characters of $\Lambda$. But according
	to \cite[Lemma 15]{thoma1964unitare}, each $\psi\in \Ch(\Lambda)$
	with $\psi\res_N = \varphi$ appears as one of those $[\Lambda:N]$ many characters in the decomposition of $\tilde{\varphi}$. It follows that the fiber of $\varphi$ under the restriction map 
	$\Ch(\Lambda)\to \Ch_{\Lambda}(N)$ is of size at most $[\Lambda:N]$. 
\end{proof}

\section{Stiffness of arithmetic groups \label{sec:Stiffness results}}

The goal of this section is to prove Theorem \ref{thm:main-stiffness} bellow and to deduce from it Theorems \ref{thm:stiffness of algebraic groups} and  \ref{thm:higher-rank-stiffness}.
\begin{thm}
\label{thm:main-stiffness} Let $U$ be a connected, simply connected nilpotent
real Lie group and $H\leq U$ a lattice. 
Let $S$ be a reductive group with compact center acting continuously by automorphisms on $U$ and let $\Gamma<S$ be a lattice preserving $H$.
Let $\mu$ be a Furstenberg measure on $\Gamma$. Then the action of $\Gamma$
on $\Ch(H)$ is $\mu$-stiff. 
\end{thm}

Due to Lemma \ref{lem:stiffness-subgroup}, there is no loss of generality in assuming $S$ to be connected.
We denote by $S_{\mathrm{nc}}$ the subgroup of $S$ generated by unipotents.
This is a semisimple group with no compact simple factors.
We denote by $S_{\mathrm{c}}$ the maximal compact normal subgroup of $S$.
We obtain a finite central extension $S_{\mathrm{nc}}\times S_{\mathrm{c}} \to S$.
Upon replacing $S$ with this central extension, 
replacing $\Gamma$ with its preimage and replacing $\mu$ with an arbitrary lift, we assume as we may that 
$S=S_{\mathrm{nc}}\times S_{\mathrm{c}}$.
Note that indeed, $\mu$ is still a Furstenberg measure on $\Gamma$.

We note that $S$ has a canonical real algebraic structure and every continuous linear representation of it is automatically real algebraic.
By the Borel Density Theorem,
the image of $\Gamma$ is Zariski dense in $S_{\mathrm{nc}}$.
Upon replacing $S_{\mathrm{c}}$ with the closure of the image of $\Gamma$,
we assume as we may that this image is dense.
We thus have that $\Gamma$ is Zariski dense in $S$. 

We let $P<S$ be a minimal parabolic and note that $P=P_{\mathrm{nc}} \times S_{\mathrm{c}}$,
where $P_{\mathrm{nc}}$ is a minimal parabolic in $S_{\mathrm{nc}}$.
Thus $S/P\simeq S_{\mathrm{nc}}/P_{\mathrm{nc}}$.
Endowing this space with the Haar measure class, it is $\Gamma$-isomorphic as a Lebesgue space to $\Pi(\Gamma,\mu)$,
by the assumption that $\mu$ is a Furstenberg measure.

\subsection{The abelian case }

We start by proving the following special case.  

\begin{prop}[{cf. \cite{furstenberg1998stiffness}}]
\label{prop:abelian-stiffness} 
Theorem \ref{thm:main-stiffness}
holds when $U$ is abelian.
\end{prop}

This formulation is merely a slight generalization of Furstenberg's original result, 
yet this generalization is essential to our discussion, so we 
give here its full proof.
We will need the following definition.

\begin{defn}
\label{def:epimorphic}Let $k$ be a field and let $\bL$ be a $k$-algebraic group. An algebraic subgroup
$\bH$ defined over $k$ is said
to be \emph{epimorphic} in $\bL$ if for any $k$-representation
$V$ of $\bL$, a vector $v\in V$ is fixed by $\bL$ as soon as it
is fixed by $\bH$.
\end{defn}

Any $k$-parabolic subgroup $\bP$ of a connected $k$-algebraic group $\bL$ is epimorphic.
To see that, suppose $\rho:\bL\to\mathbf{\GL}(V)$ is a representation
and that $v\in V$ is $\bP$-fixed. Thus the orbit map $\bL\to V$
given by $g\mapsto\rho(g)v$ factorizes to a map $\bL/\bP\to V$.
But this map must be constant as $\bL/\bP$ is a projective variety,
meaning that $v$ is $\bL$-fixed. 

The proof of Proposition~\ref{prop:abelian-stiffness}  is based on
the following measure rigidity theorem due to Mozes.

\begin{thm}[{\cite[Theorem 2]{mozes1995epimorphic}}]
\label{thm:Mozes}Let $\bG$ be an $\R$-algebraic group, let $\bL\leq\bG$
be an $\mathbb{R}$-algebraic subgroup generated by connected unipotent
$\R$-subgroups and let $\bH<\bL$ be a connected epimorphic $\mathbb{R}$-algebraic subgroup subgroup
of $\bL$. Consider a discrete subgroup $\Gamma$ of $\boldsymbol{\bG}(\R)$.
Then any probability measure on $\bG(\R)/\Gamma$ is $\bL(\R)$-invariant
as soon as it is $\bH(\R)$-invariant.
\end{thm}

\begin{proof}[Proof of Proposition~\ref{prop:abelian-stiffness}]
In this case $U\cong\R^{m}$ and $\Ch(H)\cong U/H$ is a torus which
we denote by $X$. 
Let $\nu\in \Prob(X)^{\mu}$ and consider the
corresponding boundary map given by Theorem~\ref{thm:Poisson-boundary-universal-property},
\[
\Theta:S/P\simeq S_{\mathrm{nc}}/P_{\mathrm{nc}} \simeq \Pi(\Gamma,\mu)\to \Prob(X).
\]
We will show that $\Theta$ is essentially constant, thus $\nu$ is $\Gamma$-invariant (Lemma \ref{lem:invcrit}).

We let $\dd s$ be the Haar measure on $S$, 
normalized so that it induces a probability measure on $S/\Gamma$.
We consider the quotient map $q:S \to S/P \simeq S_{\mathrm{nc}}/P_{\mathrm{nc}}$ and
set $\theta:=\Theta\circ q$.
This is an a.e defined $\Gamma$-equivariant
measurable map.

We consider the space $S\times_{\Gamma}X$, defined as $S\times X$
quotiented by the $\Gamma$-action $\gamma.\left(s,x\right)=\left(s\gamma^{-1},\gamma x\right)$,
and we let $S$ act on this space by left multiplication on the first
coordinate. 
We define the
measure $\tilde{\sigma}$ on $S\times X$ by:
\[
\tilde{\sigma}=\int_{S}\delta_{s}\times\theta_{s^{-1}}\dd s.
\]
It is easy to see that $\tilde{\sigma}$ is both $\Gamma$-invariant,
for the action considered above,
and $P$-invariant, for the restriction of the left $S$-action. Hence $\tilde{\sigma}$
descends to a $P$-invariant probability measure $\sigma$ on $S\times_{\Gamma}X$. 

We define $G=S\ltimes U$.
By the fact that every linear representation of $S$ is real algebraic, we have that $G$ is a real algebraic group.
We set $\Lambda=\Gamma\ltimes H$ and notice that this is a lattice in $G$.
In fact, we may identify
\[ G/\Lambda=(S\ltimes U)/\Lambda\simeq  
((S\ltimes U)/H)/\Gamma\simeq (S \times X)/\Gamma= S\times_{\Gamma}X.\]
Using this identification, we endow $G/\Lambda$ with the measure $\sigma$ constructed above.

We are now in the setting of theorem \ref{thm:Mozes}:
$G$ is a real algebraic group, $\Lambda\leq G$ is discrete,
$S_{\mathrm{nc}}$ is a connected real algebraic subgroup of $G$ generated
by unipotents, and $P_{\mathrm{nc}}$ is an epimorphic subgroup of $S_{\mathrm{nc}}$ that fixes
$\sigma$. We conclude that $\sigma$ is $S_{\mathrm{nc}}$-invariant. 
In turn, we get that $\tilde{\sigma}$
is $S_{\mathrm{nc}}$-invariant.
Since for $t\in S_{\mathrm{nc}}$ we have
\[
t.\tilde{\sigma}=\int_{S}\delta_{ts}\times\theta_{s^{-1}}\dd s
=\int_{S}\delta_{s}\times\theta_{s^{-1}t} \dd s.
\]
The $S_{\mathrm{nc}}$-invariance of  $\tilde{\sigma}$  implies that
$\theta_{s^{-1}t}=\theta_{s^{-1}}$ almost surely, meaning that
$\theta$, and therefore also $\Theta$, is $S_{\mathrm{nc}}$-invariant.
It follows that $\Theta:S_{\mathrm{nc}}/P_{\mathrm{nc}}\to \Prob(X)$ is essentially constant,
as desired. 
\end{proof}

\subsection{\label{subsec:Proof-of-Theorem main-stiffness}Proof of Theorem \ref{thm:main-stiffness}}

The path from the abelian case to the nilpotent case is
based on central induction of characters. 
\begin{defn}\label{def:FC}
The\emph{ FC-center }of a group $H$ is the union of all of its finite
conjugacy classes, and is denoted by $\FC(H)$. $H$ is called \emph{centrally
inductive} if every character $\varphi$, viewed as a faithful characters of $H/\ker\varphi$, is supported on $\FC\left(H/\ker\varphi\right)$.
\end{defn}

The FC-center is a characteristic subgroup that plays an important role in the character theory of countable
groups (see for example the recent result in \cite{bekka2021plancherel}),
and in greater extent for nilpotent groups:
\begin{thm}[\cite{howe1977representations,carey-moran1984nil-char}]
\label{thm:central-induction}Any finitely generated nilpotent group
is centrally inductive.
\end{thm}

We will use the following simple Lemma:
\begin{lem}\label{lem:nilpotent-normal-subgroups}
Let $H$ be a nilpotent group. Then any non-trivial normal subgroup
intersects the center non-trivially.
\end{lem}

\begin{proof}
Let $Z=Z_{1}\leq...\leq Z_{m}=H$ be
the central ascending sequence of $H$
and assume by contradiction that 
$N\lhd H$ is a non-trivial normal subgroup intersecting $Z$ trivially. 
Choose $x\in N$ which belongs
to $Z_{i}$ for a minimal $i$. 
Since $x\notin Z$, we can find $y\in H$ with $[x,y]\in N$ non-trivial,
contradicting the minimality of $i$.
\end{proof}

The proof of Theorem \ref{thm:main-stiffness} is by induction on
the dimension of $U$. Let $\nu\in \Prob(\Ch(H))^{\mu}$ ergodic.
Consider the map $k:\Ch(H)\to \Sub(Z)$ that takes each
character to the intersection of its kernel with the center $Z$ of $H$. Then $k_{*}\nu$
is an ergodic $\mu$-stationary probability measure on $\Sub(Z)$.
But $\Sub(Z)$ is countable and therefore by Lemma \ref{prop:countable-fibers-stiffness} (applied for $X=\Sub(Z)$ and $Y$ being a singleton)
$k_{*}\nu$ must be $\Gamma$-invariant, thus by ergodicity, a uniform measure on some
finite orbit $\mathcal{D}\subseteq \Sub(Z)$.

\paragraph*{Case a: $\mathcal{D}$ contains a non-trivial subgroup.}
Fix a non-trivial subgroup  $K\in \mathcal{D}$. By the assumptions on $U$, the  exponential map $\exp:\Lie(U)\to U$ is bijective. Let $\tilde{K}$ denote the Lie subgroup of $U$ corresponding to the $\R$-span of $\log(K)$. The group $Z$ is finitely generated free abelian \cite[Theorem 2.18]{raghunathan1972discrete} so that $K$ is infinite. It follows that $\tilde{K}$ is a Lie subgroup of the center of $U$ of positive dimension. Moreover $K$ is a finite index subgroup of $H\cap \tilde{K}$ and is a lattice in $\tilde{K}$.  

Let $\Gamma'$ be the kernel of the action of $\Gamma$ on $\mathcal{D}$,
endowed with the hitting measure $\mu'$. 
By Proposition~\ref{prop:hitting-measure-Poisson-boundary} $\mu'$ is a Furstenberg measure.
As $K$ is $\Gamma'$-invariant so is $\tilde{K}$ \cite[Theorem 2.11]{raghunathan1972discrete}. 
By the Zariski density of $\Gamma'$ in $S$, $\tilde{K}$ is also $S$-invariant and therefore $S$ acts on
$U_0:=U/\tilde{K}$ by continuous automorphisms. Now $H_0:=H/\left(H\cap\tilde{K}\right)$
is discrete in $U_0$ and moreover a lattice \cite[Theorem 2.1]{raghunathan1972discrete}.
By induction hypothesis, $\Ch(H_0)$ is $\mu'$-stiff. Since this
holds for each $K\in\mathcal{D}$, me may apply Proposition \ref{lem:If-kernel-acts-stiffly-on-fibers-then-also-Gamma}
to the map $k:\supp(\nu)\to\mathcal{D}$ to conclude that the action
of $\Gamma$ on $\supp(\nu)$ is $\mu$-stiff. It follows that $\nu$
is $\Gamma$-invariant. 

\paragraph*{Case b: $\mathcal{D}$ consists only of the trivial subgroup.}

By Lemma \ref{lem:nilpotent-normal-subgroups} we get in this case that $\nu$-a.e character of $H$ is faithful.
Since $H$
is a lattice in $U$, it is nilpotent, finitely generated and torsion-free
\cite[Theorem 2.18]{raghunathan1972discrete}. Therefore it is centrally
inductive, and its FC-center is its center, $Z$ (see \cite[Theorem 4.37]{finiteness-soluble-groups}).
Thus $\nu$-a.e $\chi\in \Ch(H)$ is supported on $Z$.
Writing $\varphi=\mathcal{F}_{H}(\nu)$ for the corresponding trace (Theorem~\ref{thm:fourier}),
we get that $\varphi$ is supported on $Z$:
\begin{align}\label{eq:h not in Z}
	\forall h\in H\backslash Z:\,\,\varphi(h)=\int_{\Ch(H)}\chi(h)d\nu(\chi)=0
\end{align}
The pushforward measure $r_{*}\nu$ under the $\Gamma$-equivariant
restriction map $r:\Ch(H)\to\Ch(Z) \simeq \hat{Z}$ is a $\mu$-stationary probability measure.
$Z$ is a lattice in the center of $U$ \cite[Proposition 2.17]{raghunathan1972discrete}.
Since $Z$ is $\Gamma$-invariant and the center of $U$ is $S$-invariant,
we get by Proposition \ref{prop:abelian-stiffness} that 
$r_{*}\nu$ is $\Gamma$-invariant. The trace corresponding to $r_{*}\nu$
via the Fourier transform is $\varphi\res_{Z}$  and so we see that 
\begin{align}\label{eq:h in Z}
	\forall\gamma\in \Gamma,\;\; h\in Z:  \varphi^{\gamma}(h)=\varphi(h)
\end{align}
From (\ref{eq:h not in Z}) and (\ref{eq:h in Z}) we see that $\varphi$ and therefore $\nu$ is $\Gamma$-invariant.

\subsection{Proof of Theorem \ref{thm:higher-rank-stiffness}}

By pulling back via a central extension, 
replacing $\Lambda$ with its preimage and $\mu$ with an arbitrary lift 
(which is necessarily still a Furstenberg measure),
we assume as we may that $G$ is algebraically simply connected. We may also assume that $G$ is connected due to Lemma \ref{lem:stiffness-subgroup}.
We consider an action of $\Lambda$ on a finitely generated virtually nilpotent group $\Gamma$. 
By Corollary~\ref{cor:stiff commen} we assume as we may that $\Gamma$ is actually nilpotent.
Recall that for a finitely generated nilpotent group the set of torsion elements
is a characteristic finite subgroup.
Applying again Corollary~\ref{cor:stiff commen},
we assume as we may that $\Gamma$ is torsion free.

There exists a simply connected nilpotent Lie
group $U$ that has $\Gamma$ embedded as a lattice \cite[Theorem 2.18]{raghunathan1972discrete}.
The group of continuous automorphisms $\Aut(U)$ is real algebraic
subgroup of $\GL\left(\Lie(U)\right)$. 

If the action is finite, then the statement holds by Lemma \ref{lem:finite-is-stiff}.
Otherwise, the image of the action map $\rho:\Lambda\to \Aut(H)$ is
unbounded in $\Aut(U)$ because $\Aut(H)$ is discrete. Margulis's super-rigidity
theorem \cite[Theorem 16.1.4]{morris2001introduction} then guarantees
a finite index subgroup $\Lambda'\leq\Lambda$ and a continuous representation
$\hat{\rho}:G\to \Aut(U)$ such that $\hat{\rho}\res_{\Lambda'}=\rho$. 
By Theorem \ref{thm:main-stiffness},
$\Lambda'$ acts $\mu'$-stiffly on $\Ch\left(\Gamma\right)$ where $\mu'$
is the hitting measure coming from $\mu$. It follows by Lemma \ref{lem:stiffness-subgroup} that the
action of $\Lambda$ on $\Ch(\Gamma)$ is $\mu$-stiff.

\subsection{\label{subsec:The-amenable-radical}The amenable radical of arithmetic
groups}

Let $\Lambda$ be an arithmetic subgroup of a  $\Q$-algebraic group  $\bG$. For a $\Q$-algebraic subgroup $\bH$ of $\bG$ denote $\Lambda_\bH:=\Lambda\cap \bH(\Q)$. Then $\Lambda_\bH$ is clearly an arithmetic subgroup of $\bH$. If $\bar{\bG}$ is a quotient of $\bG$ by a $\Q$-algebraic normal subgroup, denote $\Lambda_{\bar{\bG}}$ for the image of $\Lambda$ under the quotient map $\bG \to \bar{\bG}$. Then $\Lambda_{\bar{\bG}}$ is an arithmetic subgroup of $\bar{\bG}$ \cite[Theorem 4.1]{platonov-rapinchik-1993}.

\begin{lem}
\label{lem:arithmetic-radical-finite-center}Let $\Lambda$ be an
arithmetic subgroup of a $\Q$-algebraic semisimple group $\bG$.
Then $\Rad(\Lambda)$ is finite, and it is moreover trivial assuming
that $\bG$ has no finite normal subgroups, and no $\Q$-simple $\R$-anisotropic factors. 
\end{lem}

\begin{proof} 

Assume first that $\bG$ has no $\Q$-simple $\R$-anisotropic factors, in which case $\Lambda$ is Zariski dense in $\bG(\R)$ by Borel's density theorem. Let $\bH$ denote the $\R$-Zariski closure of $\Rad(\Lambda)$ in $\bG$. Then $\bH(\R)$ is an amenable normal subgroup of $\bG(\R)$ and thus must be compact \cite[\S4.1]{zimmer2013ergodic}. But $\Rad(\Lambda)$ is discrete and so it must be finite, making $\bH$ finite as well. Hence, if $\bG$ is in addition assumed to have no  finite normal subgroups, then  $\Rad(\Lambda)$ is trivial. 

Now allow $\bG$ to have $\Q$-simple $\R$-anisotropic factors. Let $\bG_c$ denote the product of all $\R$-anisotropic $\Q$-simple factors and denote $\bar{\bG}:=\bG/\bG_{c}$. Then the image $\bar{\Lambda}$ of $\Lambda$ under this quotient is an arithmetic subgroup of $\bar{\bG}$. By the previous paragraph, $\Rad(\bar{\Lambda})$ is finite and therefore also $\overline{\Rad(\Lambda)}$ as it is clearly contained in $\Rad(\bar{\Lambda})$. But $\Lambda_{\bG_c}$ is finite as it is a discrete subgroup of the compact group  $\bG_c(\R)$, so it follows that $\Rad(\Lambda)$ is finite. 
\end{proof}

We now consider a general arithmetic subgroup $\Lambda$ of some $\Q$-algebraic group  $\bG$. Let $\bR$ be the solvable radical of $\bG$, and consider the semisimple group $\bS=\bG/\bR$. Observe that the size of any finite normal subgroup of $\bS$ is bounded from above by $|\bZ|\cdot [\bS:\bS^\circ]$ where $\bS^\circ$ is the identity connected component of $\bS$ and $\bZ$ is the (finite) center of $\bS^\circ$. It follows that $\bS$ admits a unique maximal finite normal subgroup, which coincides with the center when $\bS$ is connected.  Let
$\bar{\bS}$ be the quotient of $\bS$ by all $\R$-anisotropic $\Q$-simple
factors, as well as by the remaining maximal finite normal subgroup. We consider the arithmetic subgroup $\Lambda_\bR$  of $\bR$ and $\Lambda_{\bar{\bS}}$  of  $\bar{\bS}$. 
 
\begin{prop}
\label{prop:amenable radical of arithmetic groups}Let $\Lambda$
be an arithmetic subgroup of $\bG$. Then:
\begin{enumerate}
\item $\Lambda_{\bR}$ is a finite index subgroup of $\Rad(\Lambda)$ 
\item $\Lambda_{\bar{\bS}}$ is equal to $\Lambda/\Rad(\Lambda)$. 
\end{enumerate}
\end{prop}

\begin{proof}
It is clear that $\Lambda_{\bR}$ is contained in $\Rad(\Lambda)$.  $\Lambda_\bS=\Lambda/\Lambda_{\bR}$
is an arithmetic subgroup of the semisimple group $\bS$ so that $\Rad\left(\Lambda_{\bS}\right)$ is finite by Lemma \ref{lem:arithmetic-radical-finite-center}.
Item 1 follows as 
$\Rad(\Lambda)/\Lambda_{\bR}$ is contained in the finite group $\Rad\left(\Lambda/\Lambda_{\bR}\right)=\Rad\left(\Lambda_{\bS}\right)$.

For 2, consider the quotient map $\sigma:\bG\to\bar{\bS}$, and
let $K\leq\Lambda$ be the kernel of $\sigma\res_{\Lambda}$. $\sigma(\Rad(\Lambda))$
is a normal amenable subgroup of the arithmetic subgroup $\Lambda_{\bar{\bS}}=\sigma(\Lambda)$
of $\bar{\bS}$. By Lemma \ref{lem:arithmetic-radical-finite-center},  $\sigma(\Rad(\Lambda))$ is trivial meaning that $\Rad(\Lambda)\subseteq K$.
But $K$ is clearly normal and amenable and so we actually have $\Rad(\Lambda)=K$.
It follows that $\Lambda/\Rad(\Lambda)=\Lambda_{\bar{\bS}}$. 
\end{proof}

As any solvable arithmetic group is polycyclic \cite{mal1951certain} we get the following as a consequence:
\begin{cor}\label{cor:amenable radical is virtually polycyclic}
	The amenable radical of an arithmetic group is virtually polycyclic.
\end{cor}

\subsection{\label{subsec:Proof-of-Theorem-arithmetic-stiffness}Proof of Theorem
\ref{thm:stiffness of algebraic groups}}

The following lemma and its proof can be found in a more
general form in \cite{carey-moran1984nil-char}:
\begin{lem}[{\cite[Lemma 2.2]{carey-moran1984nil-char}}]
\label{lem:H/N dual acts transitively on fibers}Let $H$ be a discrete
group, $N\leq H$ a normal subgroup, and assume that $A:=H/N$ is
abelian. Consider the continuous action of $\hat{A}$ on $\Ch(H)$ given by pointwise multiplication: $(\chi,\varphi)\mapsto \chi\cdot\varphi$. Then this action preserves the fibers of the restriction
map $\Ch(H)\to \Ch_{H}(N)$ and acts on each of them transitively. 
\end{lem}

We first note that it is enough to show the statement of Theorem \ref{thm:stiffness of algebraic groups} for when $\bG$ is connected. The reason is that it is not hard to see that the amenable radical of two commensurable subgroups is commensurable. Thus the non-connected case is reduced to the connected case, using Lemma \ref{lem:stiffness-subgroup} and Corollary~\ref{cor:stiff commen}.

Let $\bG$ be a connected $\Q$-algebraic group with solvable radical $\bR$,
unipotent radical $\bU$ and semisimple part $\bS$.  By general theory $\bR/\bU$ is abelian and is centralized by $\bS$.
Let $\Lambda$ be an arithmetic subgroup of $\bG$ and let $\mu$ be a Furstenberg measure on $\Lambda$. 
We use the notations introduced in \S\ref{subsec:The-amenable-radical}.
By the classical theorem of Borel and Harish-Chandra $\Lambda_{\bS}$ is a lattice
in the real semisimple Lie group $S=\bS(\R)$.
We denote by $\bar{\mu}$ the measure on $\Lambda_{\bS}$ pushed from $\mu$.
Clearly, this is a Furstenberg measure.
By Mal'cev theory,
$\Lambda_{\bU}$ is a lattice in the simply connected nilpotent Lie
group $U=\bU(\R)$. 
Thus, by Theorem \ref{thm:main-stiffness}, the action of $\Lambda_\bS$ on $\Ch(\Lambda_{\bU})$
is stiff and by Proposition \ref{prop:Ch(N) stiff then also Ch_H(N)} so
is its action on $\Ch_{\Lambda_{\bR}}(\Lambda_{\bU})$. Consider the restricting
map $r:\Ch(\Lambda_{\bR})\to \Ch_{\Lambda_{\bR}}(\Lambda_{\bU})$. The
group $A:=\Lambda_{\bR}/\Lambda_{\bU}$ is abelian and it is centralized
by $\Lambda_{\bS}$. Hence by Lemma \ref{lem:H/N dual acts transitively on fibers}, the compact group $\widehat{A}$
acts transitively on each fiber of $r$, commuting with the action
of $\Lambda_{\bS}$. 
Using Proposition~\ref{prop:X->X/K is relatively stiff} with $X=\Ch(\Lambda_{\bR})$, 
$Y=\Ch_{\Lambda_{\bR}}(\Lambda_{\bU})$ and $K=\hat{A}$, we get that $r$ is relatively stiff with respect to the $\Lambda_{\bS}$-action.
We deduce by Lemma \ref{lem:relaltively-stiff} that the $\Lambda_{\bS}$-action on $\Ch(\Lambda_{\bR})$ is stiff
and, using Lemma~\ref{lem:stationarity-of-quotients}, we get that the $\Lambda$-action on $\Ch(\Lambda_{\bR})$ is stiff.
Finally, using Proposition \ref{prop:amenable radical of arithmetic groups},  
we get by Corollary~\ref{cor:stiff commen} that 
the $\Lambda$-action on $\Ch(\Rad(\Lambda))$ is stiff.

\section{Charmenability of arithmetic groups \label{sec:Charmenability}}

\subsection{A criterion for charmenability}

\begin{defn}
	Let $\Lambda$ be a countable group and let $\Omega$ be a Lebesgue $\Lambda$-space. 
	\begin{enumerate}
		\item $\Omega$ is called \emph{metrically ergodic}, if for any action
		of $\Lambda$ by isometries on a separable metric space $Z$, any
		Lebesgue $\Lambda$-map $\Omega\to Z$ is essentially constant. 
		\item $\Omega$ is called \emph{Zimmer amenable} if for any Borel $\Lambda$-space $W$,
		with an essentially surjective Borel $\Lambda$-map $p:W\to \Omega$ whose
		fibers admit compact convex structures on which $\Lambda$ acts affinely,
		there exists a Lebesgue $\Lambda$-map which is a section of $p$. 
		\item $\Omega$ is called a \emph{$\Lambda$-boundary} if it is both metrically ergodic and Zimmer amenable.
	\end{enumerate}
\end{defn}

The following is a strengthening of a theorem by Kaimanovich \cite{kaimanovich2003double}.
\begin{thm}[{\cite[Theorem~2.7]{bader2014boundaries}}]
	\label{fact:kaimanovich} Let $\Lambda$ be a countable group and
	let $\mu$ be an admissible probability measure on $\Lambda$. Then
	the Furstenberg-Poisson boundary $\Pi(\Lambda,\mu)$ is a $\Lambda$-boundary. 
\end{thm}

We will make use of the following simple lemma:

\begin{lem}
	\label{lem:amenable-metrically-ergodic}Let $\Lambda$ be a countable group, 
	let $N\lhd \Lambda$ be a normal amenable subgroup and set $\Gamma=\Lambda/N$. If $B$ is a $\Gamma$-boundary then it is also a $\Lambda$-boundary.
\end{lem}

\begin{proof}
	We show that $B$ is metrically ergodic as a $\Lambda$-space assuming it is such as a $\Gamma$-space (in fact regardless of whether $N$ is amenable), and that $B$ is Zimmer amenable as a $\Lambda$-space assuming it is such as a $\Gamma$-space. 
	
	Suppose $\Lambda$ acts by isometries on a separable metric space
	$Z$, and let $f:B\to Z$ be a Lebesgue $\Lambda$-map. Since $N$ acts
	trivially on $B$, the essential image of $f$ is supported on $N$-fixed
	points of $Z$. We may therefore view $f$ as a Lebesgue $\Gamma$-map
	$B\to Z^{N}$. By metric ergodicity of $B$ as a $\Gamma$-space, $f$ must be essentially constant. This shows that $B$ is metrically ergodic as $\Lambda$-space
	
	Consider now a Borel $\Lambda$-space $W$ with an essentially
	surjective Borel $\Lambda$-map $p:W\to B$, whose fibers admit compact
	convex structures on which $\Lambda$ acts affinely. Each fiber of $p$ is $N$-invariant, and therefore admits $N$-fixed points. It follows that the map $p\res_{W^N}:W^N\to B$ is an essentially surjective Borel $\Gamma$-map with compact convex fibers on which $\Gamma$ acts affinely.  By Zimmer amenability of $B$
	as a Lebesgue $\Gamma$-space, there exists a Lebesgue $\Gamma$-map which is a section of $p\res_{W^N}$, and this map is obviously also a Lebesgue $\Lambda$-map and a section of $p$.
\end{proof}

The proof of Theorem \ref{thm:Main-charmenability} is based on a criterion for charmenability established in \cite{bader2022charmenability}. We shall first state it, and then explain the notions involved.

\begin{criterion}[{\cite[Proposition 4.20]{bader2022charmenability}}]
	\label{crit:criterion for charmenability}Let $\Lambda$ be a countable
	group and denote $\Gamma:=\Lambda/\Rad(\Lambda)$. Assume that $\Gamma$ admits a boundary $B$ satisfying the following two conditions:
	\begin{enumerate}
		\item Every separable, ergodic, faithful $(\Gamma,L^{\infty}(B))$-von Neumann
		algebra $(M,E)$ is either invariant or $\Gamma$-singular.
		\item Every Lebesgue $\Lambda$-map $B\to \PD_{1}(\Rad(\Lambda))$ is essentially
		constant. 
	\end{enumerate}
	Then $\Lambda$ is charmenable.
\end{criterion}

Note that when $\Lambda$ is arithmetic, the first condition of the criterion deals only with the semisimple part. The fact that it holds in the higher rank setting is essentially established in \cite{boutonnet2021stationary} and \cite{bader2022charmenability}, see Proposition \ref{prop:first-condition-criterion} bellow. As the main focus of this work is in cases where the amenable radical is non-trivial, it is the second condition that is of our main interest. We shall nevertheless  briefly clarify the first condition, but emphasize that the involved notions do not play an important conceptual role in this work. See \cite{bader2022charmenability} for further details.

\begin{defn}
	Let $A$ be a C*-algebra, and let  $\mathcal{S}(A)$ denote the space of states on $A$. $\phi_1,\phi_2\in \mathcal{S}(A)$ are called \emph{singular} if $\|\phi_1-\phi_2\|=2$. 
	
	Given a Lebesgue space $\Omega$, two Lebesgue maps $\theta_1,\theta_2:\Omega \to \mathcal{S}(A) $ are called \emph{singular} if $\theta_1(\omega)$ and $\theta_2(\omega)$ are singular for a.e $\omega\in \Omega$.  
	
	Two unital completely-positive maps $E_1,E_2:A \to L^\infty(\Omega)$ are called \emph{singular} if the corresponding dual maps $\theta_1,\theta_2:\Omega\to \mathcal{S}(A)$ (given by $\theta_i(\omega)(a)=E_i(a)(\omega)$) are singular.
	 
	Given a separable von Neumann algebra  $M$, two normal unital completely-positive maps $E_1,E_2:M\to L^\infty(\Omega)$ are called \emph{singular} if, $M$ admits a separable unital C*-subalgebra $A$ such that the restrictions $E_1\res_{A},E_2\res_{A}$ are singular.
\end{defn}

\begin{defn}
	Let $\Lambda$ be a countable group. 
	
	A \emph{$\Lambda$-von Neumann algebra} is a von Neumann algebra $N$
	endowed with a ultraweak continuous action of $\Lambda$ by automorphisms. $N$ is said to be \emph{ergodic} if the subalgebra of fixed point $N^\Lambda$ is trivial (i.e  $=\C$). 	
	
	Given a $\Lambda$-von Neumann algebra $N$,
	a \emph{$(\Lambda,N)$-von Neumann algebra} is a pair $\left(M,E\right)$
	where $M$ is a $\Lambda$-von Neumann algebra, and $E$ is a unital completely-positive
	$\Lambda$-equivariant map $E:M\to N$. $(M,E)$ is said \emph{faithful} or \emph{extremal}
	if $E$ satisfies the corresponding properties. If $E(M)\subset N^\Lambda$,
	we say that $(M,E)$ is\emph{ $\Lambda$-invariant}. 	
	 
	Assume $N=L^\infty(\Omega)$ where $\Omega$ is a Lebesgue $\Lambda$-space. Then $E$ is said to be \emph{$\Lambda$-singular} if the normal unital completely-positive maps $E_g:x\in M\mapsto E(gx)\in L^\infty(\Omega)$ are pairwise singular.
\end{defn}

The only important example for our concern is when $N$ is the commutative
von Neumann algebra of $L^{\infty}$-functions on the Furstenberg-Poisson boundary
of $\Lambda$ w.r.t some measure $\mu$. In such case $(\Lambda,N)$-von Neumann algebras may be seen as a non-commutative version of the boundary maps appearing in Theorem \ref{thm:Poisson-boundary-universal-property}.

\subsection{Proof of Theorem~\ref{thm:Main-charmenability}} \label{subsec:thm:main-charmenability}
 
The first condition of Criterion \ref{crit:criterion for charmenability} was essentially shown to hold for arithmetic subgroups of semisimple groups of higher rank:

\begin{prop}\label{prop:first-condition-criterion}
	Let $\Gamma$ be an arithmetic subgroup of a connected  
$\Q$-algebraic group $\bG$, and let $B$ be its Poisson boundary w.r.t a Furstenberg measure $\mu$. Assume that $\bG$ is $\Q$-simple, center-free, without $\R$-anisotropic $\Q$-simple factors, and with $\R$-rank at least $2$. Then the first condition of Criterion \ref{crit:criterion for charmenability} holds true.
\end{prop} 
 
\begin{proof}
	Let $G=\bG(\R)$. By Lemma \ref{lem:arithmetic-radical-finite-center}, $\Rad(\Gamma)$ is trivial. The intersection of  $\Gamma$ with any compact simple factor of $G$ is central, hence trivial. We may therefore assume that $G$ has no compact simple factors. Let $(M,E)$ be a separable, ergodic, faithful $(\Gamma,L^\infty(B))$-von Neumann algebra which is not $\Gamma$-invariant. We explain how to deduce singularity of $E$ from the work of \cite{boutonnet2021stationary} or of \cite{bader2022charmenability}, depending on whether $G$ is simple or a product of a few simple factors.

	Assume first $G$ is simple. Let  $P$ be a minimal parabolic subgroup of $G$ so that $B=G/P$.
	Then by \cite[Theorem B]{boutonnet2021stationary} there is a proper closed subgroup $P\leq Q \lneq G$ and a unital normal $\Gamma$-equivariant $*$-embedding $\iota:L^\infty(G/Q)\to M$ such that $E\circ\iota$ is equal to the natural embedding $p_Q^*: L^\infty(G/Q)\to L^\infty(G/P)$. Note that the theorem in this reference is formalized in a somewhat dual way, though it is equivalent (see \cite[Theorem E]{houdayer2021noncommutative}). Let $M_0=\iota(L^\infty(G/Q))$ and consider $(M_0,E\res_{M_0})$ as a $(G,L^\infty(G/P))$-von Neumann algebra. It is not $G$-invariant as $Q$ is proper, and it is extremal because $M_0$ is ergodic (see \cite[Lemma 4.16]{bader2022charmenability}). Due to \cite[Example 2.14, Lemma 6.2]{bader2022charmenability}, the singularity criterion of \cite[Proposition 4.15]{bader2022charmenability} applies so that $(M_0,E\res_{M_0})$ is $G$-singular. It is therefore in particular $\Gamma$-singular. Since $M$ contains $M_0$ as a subalgebra, it follows that $(M,E)$ is singular as well (see \cite[Remark 4.12]{bader2022charmenability}). 
	
	Now assume $G$ has at least two (non-compact) simple factors. Then, in the terminology of \cite{bader2022charmenability}, $\Gamma$ is an arithmetic group of product type and is therefore charmenable by \cite[Proposition 6.1]{bader2022charmenability}. More specifically, it is transparent that the proof of this proposition all boils down to showing that $E$ and $E_\gamma$ are singular for any non-central element $\gamma\in \Gamma$. But in our case  $\Gamma$ is center-free, and so $E$ is $\Gamma$-singular. 
\end{proof}

The dynamical study of \S\ref{sec:Stiffness results}, which eventually lead to the proof of Theorem \ref{thm:stiffness of algebraic groups}, was aimed towards the establishment of the second condition of Criterion \ref{crit:criterion for charmenability} for arbitrary arithmetic groups:
 
\begin{prop}\label{prop:second-condition-criterion}
	Let $\Lambda$ be an arithmetic group and let $(B,\nu_B)$ be its Poisson boundary w.r.t a Furstenberg measure $\mu$. Then the second condition of Criterion \ref{crit:criterion for charmenability} holds true.
\end{prop}

\begin{proof}
	Let $\theta:B\to \PD_{1}(\Rad(\Lambda))$ be a Lebesgue $\Lambda$-map.
	Since $\Rad(\Lambda)$ acts trivially on $B$, the essential image
	of $\theta$ consists of traces, namely, we may view $\theta$ as a Lebesgue
	$\Lambda$-map $\theta:B\to \Tr(\Rad(\Lambda))$. 
	By pushing forward $\nu_B$ through $\theta$ and taking its barycenter, we get a $\mu$-stationary trace $\varphi\in \Tr(\Rad(\Lambda))^\mu$. The $\mu$-stationary probability measure on $\Ch(\Rad(\Lambda))$ corresponding to $\varphi$ (Theorem \ref{thm:fourier}) is $\Lambda$-invariant by Theorem~\ref{thm:stiffness of algebraic groups}. 
	It follows that $\varphi$ is $\Lambda$-invariant, which means that $\theta$
	is essentially constant by Lemma~\ref{lem:invcrit}.
\end{proof}

\begin{proof}[Proof of Theorem \ref{thm:Main-charmenability}]
Let 
$\bG$ be a $\Q$-algebraic group with solvable radical $\bR$ and $\bS=\bG/\bR$, and let $\Lambda_{\bS}$ be the image of  the arithmetic subgroup $\Lambda$ under the quotient map $\bG\to \bS$.  
We first observe that if $\bS$ has more than one $\R$-isotropic $\Q$-simple 
factor then $\Lambda_\bS$ has a normal subgroup which is neither amenable nor co-amenable.
Also, if the real rank of $\bS$ is $1$ then $\Lambda_\bS$ has such a normal subgroup, as it is relatively hyperbolic. In both cases we see that $\Lambda$ admits a normal subgroup which is neither amenable nor co-amenable which implies that $\Lambda$ is not charmenable \cite[Propositions 3.3]{bader2022charmenability}.
This shows that the conditions imposed on $\bS$ in the statement of the theorem are necessary for charmenability of $\Lambda$. 
We now turn to show that they are sufficient as well. Note that as $\Lambda$ is an arbitrary arithmetic subgroup of $\bG$, this will in particular show that either all arithmetic subgroups of $\bG$ are charmenable, or else none of them are. 

If the real rank of $\bS$ is
$0$ then $\bS(\R)$ is compact in which case $\bG(\R)$
is amenable, hence $\Lambda$ is amenable and a fortiori charmenable.
We thus assume from now on that $\bS$ has at most one $\R$-isotropic $\Q$-simple  factor
and that the real rank of $\bS$ is at least
$2$, and we argue to show that $\Lambda$ is charmenable.
We will do so by showing that it satisfies the first condition of Criterion~\ref{crit:criterion for charmenability} (the second condition holds in general by Proposition \ref{prop:second-condition-criterion})

Let $\bar{\bS}$ denote the quotient of $\bS$ by all $\R$-anisotropic $\Q$-simple factors, as well as by the remaining center. Let $\Lambda_{\bar{\bS}}$ denote the image of $\Lambda$ under the quotient map $\bG\to \bar{\bS}$.  Fix a Furstenberg measure
$\mu$ on $\Lambda$ and let $\bar{\mu}$ be the pushforward measure
onto $\Lambda_{\bar{\bS}}$, which is a Furstenberg measure as well. 
The Furstenberg-Poisson boundary of $\Lambda$ w.r.t $\mu$
is the same as the Furstenberg-Poisson boundary $\Lambda_{\bar{\bS}}$ w.r.t $\bar{\mu}$
and we denote it by $B$. By Theorem~\ref{fact:kaimanovich}, $B$ is
a $\Lambda_{\bar{\bS}}$-boundary and therefore
also a $\Lambda$-boundary by Lemma \ref{lem:amenable-metrically-ergodic}.
By Proposition \ref{prop:amenable radical of arithmetic groups}, we have that 
$\Lambda/\Rad(\Lambda)=\Lambda_{\bar{\bS}}$. In particular, this is an arithmetic subgroup of $\bar{\bS}$. This first condition of Criterion \ref{crit:criterion for charmenability} therefore follow from  Proposition \ref{prop:first-condition-criterion}. 
\end{proof}

\section{Properties of charmenable groups}\label{sec:cor-of-charm}

This section is dedicated for proving Propositions \ref{prop:charmenable dichotomy} and \ref{prop:charmenable dichotomy prop (T)} from the introduction. Here too, $\Lambda$ denotes a countable group.

\subsection{Amenable traces}
Let $\pi:\Lambda\to \mathcal{U}(\mathcal{H})$ be a representation. A sequence of vectors $v_n\in \mathcal{H}$ of norm $1$ is said to be \emph{almost $\Lambda$-invariant} if:
\[
\forall g\in \Lambda:\quad \|\pi(g)v_n-v_n\|\to 0
\]
If such a sequence exists, we say that $\pi$ has \emph{almost invariant vectors}.
The representation $\pi$ is called \emph{amenable} if $\pi\otimes \bar{\pi}$ has almost invariant vectors, where $\bar{\pi}$ is the dual representation.  See \cite{bekka1990amenable} for more on this notion, including properties and equivalent definitions. A trace $\varphi$ on $\Lambda$ is called \emph{amenable} if its GNS representation is amenable. This should not be confused with the stronger notion of a von Neumann amenable trace given in Definition \ref{def:vNamenable}. The second condition in the definition of charmenability implies the dichotomy stated on traces:

\begin{prop}[\protect{\cite[Proposition 3.2]{bader2022charmenability}}]\label{prop:trace-dichotomy}
	Assume that every character of $\Lambda$ is either von Neumann amenable, or supported on $\Rad{(\Lambda)}$. Then any trace of $\Lambda$ is either amenable or supported on $\Rad({\Lambda})$.
\end{prop}

Note that a representation $\pi$ of a product of groups $\Gamma_{1}\times\Gamma_{2}$
may have no almost invariant vectors even though both  restrictions
$\pi\res_{\Gamma_{1}}$, $\pi\res_{\Gamma_{2}}$ have. Indeed,
if $\pi_{i}$ is a representation of $\Gamma_{i}$ without almost
invariant vectors then we may extend each $\pi_{i}$ to $\Gamma_{1}\times\Gamma_{2}$
trivially, and consider $\pi=\pi_{1}\oplus\pi_{2}$. The following
proposition shows that this is the only obstruction. 
\begin{prop}
	\label{prop:inv vectors of commuting subgroups}
	Suppose that  $\Gamma_{1},\Gamma_{2}$
	are two subgroups of $\Lambda$ which commute with each other and satisfy $\Lambda=\Gamma_1\Gamma_2$. Let $\pi:\Lambda\to\mathcal{U}(\mathcal{H})$ be a representation of $\Lambda$.    Assume
	that the restrictions $\pi\res_{\Gamma_{1}}$ and $\pi\res_{\Gamma_{2}}$
	admit almost invariant vectors, whereas $\pi$ does not. Then $\pi$
	decomposes as a direct sum $\pi=\pi_{1}\oplus\pi_{2}$ such that $\pi_{1}\res_{\Gamma_{1}}$
	and $\pi_{2}\res_{\Gamma_{2}}$ admit almost invariant vectors but
	$\pi_{1}\res_{\Gamma_{2}}$ and $\pi_{2}\res_{\Gamma_{1}}$ do not. 
\end{prop}

\begin{proof}
	We first recall a general fact about almost invariant vectors. Consider
	a countable group $\Gamma$ and fix an admissible probability measure
	$\mu$ on $\Gamma$ satisfying $\mu(\gamma^{-1})=\mu(\gamma)$ for
	all $\gamma\in\Gamma$. Given a unitary representation $\rho:\Gamma\to\mathcal{U}(\mathcal{H})$,
	there is the associated Markov operator:
	\[
	\mathcal{B}(\mathcal{H})\ni A_{\mu}:\,\,\,\,\,\,\,v\mapsto\sum_{\gamma\in\Gamma}\mu(\gamma)\rho(\gamma)v
	\]
	Then $A_{\mu}$ is a positive contractive operator, and its norm is
	$1$ if and only if $\rho$ has almost invariant vectors \cite[Proposition 7.1.4]{peterson2013notes}.
	We apply this reasoning to $\Gamma_{1}$ and $\Gamma_{2}$, that is
	we fix generating probability measures on each of those subgroups,
	and consider the corresponding Markov operators denoted by
	$A_{1}$ and $A_{2}$. Then $1$ is in the spectrum of both $A_{1}$
	and $A_{2}$, and yet, there is some $\epsilon>0$ such that:
	\begin{equation}
		\max_{i=1,2}\|A_{i}v-v\|>\epsilon,\quad \forall v\in\mathcal{H}\,\text{ with }\,||v||=1\label{eq:corner is missing}
	\end{equation}
	
	For $i=1,2$, let $p_{i}\in\mathcal{B}(\mathcal{H})$ denote the
	spectral projection of $A_{i}$ associated with the interval $[1-\epsilon,\epsilon]$,
	and denote by $\mathcal{K}_{i}:=p_{i}\mathcal{H}\subseteq\mathcal{H}$
	the corresponding subspace. Then $\mathcal{K}_{i}$ is an $A_{i}$-invariant
	subspace which contains almost $\Gamma_{i}$-invariant vectors, whereas
	its orthogonal complement $\mathcal{K}_{i}^{\bot}$ does not. Moreover,
	it follows from Eq. (\ref{eq:corner is missing}) that $\mathcal{K}_{1}$
	and $\mathcal{K}_{2}$ are orthogonal subspaces i.e $\mathcal{K}_{2}\subseteq\mathcal{K}_{1}^{\bot}$.
	Finally note that $\mathcal{K}_{1}$ is $\Gamma_{2}$-invariant (and
	vice versa). Indeed, this can be seen as a consequence of the fact
	that von Neumann algebras are closed under taking spectral projections, so that $p_{1}$ is in the double commutant of $A_{1}$ and thus commutes with $\pi(\Gamma_2)$. 
	
 	Thus for each $\gamma_{1}\in\Gamma_{1}$, $\pi(\gamma_{1})\mathcal{K}_{1}^{\bot}$
	is a $\Gamma_{2}$-invariant subspace that contains $\mathcal{K}_{2}$.
	The intersection $\mathcal{H}_{2}:=\bigcap_{\gamma_{1}\in\Gamma_{1}}\pi(\gamma_{1})\mathcal{K}_{1}^{\bot}$
	is therefore a $\Gamma_{1}\times\Gamma_{2}$-invariant subspace which
	contains $\mathcal{K}_{2}$, and as a result admits almost $\Gamma_{2}$-invariant
	vectors. $\mathcal{H}_{2}$ however does not admit almost $\Gamma_{1}$-invariant
	vectors because it is contained in $\mathcal{K}_{1}^{\bot}$. The
	orthogonal complement $\mathcal{H}_{1}:=\mathcal{H}_{2}^{\bot}$ on
	the other hand has almost $\Gamma_{1}$-invariant vectors because it contains
	$\mathcal{K}_{1}$, but does not have almost $\Gamma_{2}$-invariant
	vectors because it is contained in $\mathcal{K}_{2}^{\bot}$. The
	subrepresentations $\pi_{1},\pi_{2}$ corresponding to the invariant
	subspaces $\mathcal{K}_{1}$,$\mathcal{K}_{2}$ therefore give the
	desired decomposition. 
\end{proof}

Recall that to any positive-definite function $\varphi$ corresponds a GNS triplet $(\mathcal{H},\pi,\xi)$. If  $\varphi$ is moreover a trace then $\pi$ extends to a representation $\pi\times \rho:\Lambda\times \Lambda \to \mathcal{U}(\mathcal{H})$ such that $\pi(g)\xi=\rho(g^{-1})\xi$ for all $g\in G$, see Example \ref{example:traces-pd}. In particular $\varphi(g)=\langle \pi(g)\xi,\xi \rangle=\langle \xi,\rho(g)\xi \rangle$. We call $\rho$ the \emph{right GNS representation of $\varphi$}. To avoid confusion, we shall call $\pi$ the \emph{left GNS representation}.

\begin{prop}
	\label{prop:almost-inv-vectors-of-traces}Let $\varphi$ be a trace
	of $\Lambda$, and let $\pi$ and $\rho$ be the left and right GNS
	representations of $\varphi$. If $\pi$ admits almost invariant vectors
	then so does $\pi\times\rho:\Lambda\times\Lambda\to\mathcal{U}(\mathcal{H})$. 
\end{prop}

\begin{proof}
	Let $(\mathcal{H},\pi,\xi)$ be a GNS triplet associated with $\varphi$.
	Observe that:
	
	\begin{equation}
		\varphi(g)=\left\langle \xi,\rho(g)\xi\right\rangle =\left\langle \rho(g)^{*}\xi,\xi\right\rangle \label{eq:trace-rho-pi}, \quad g\in \Lambda
	\end{equation}
	Thus $\pi$ is equivalent to the adjoint representation $\bar{\rho}$,
	by the uniqueness of the GNS. Since $\pi$ admits almost invariant
	vectors, then so does $\rho$. Assuming by contradiction that $\pi\times\rho$
	does not admit almost invariant vectors implies, by Proposition \ref{prop:inv vectors of commuting subgroups},
	that there is a subrepresentation $\pi_{1}\times\rho_{1}$ of $\pi\times\rho$
	such that $\pi_{1}$ has almost invariant vectors but $\rho_{1}$
	does not. Let $\mathcal{H}_{1}$ denote the subspace corresponding
	to this subrepresentation. Since $\xi$ is cyclic with respect to
	$\pi$ as well as to $\rho$, its projection onto $\mathcal{H}_{1}$
	is non-zero and we denote its normalization by $\xi_{1}\in\mathcal{H}_{1}$.
	It is not hard to see that $\xi_{1}$ must be cyclic with respect
	to the representations $\pi_{1}$ and $\rho_{1}$. 
	
	Let $\varphi_{1}$ denote the positive-definite function corresponding
	to the GNS triplet $(\mathcal{H}_{1},\pi_{1},\xi_{1})$. Then a similar
	computation to Eq. (\ref{eq:trace-rho-pi}) shows that $(\mathcal{H}_{1},\bar{\rho}_{1},\xi_{1})$
	is a GNS triplet associated with $\varphi_{1}$. By uniqueness of
	the GNS we conclude that $\pi_{1}$ and $\bar{\rho}_{1}$ are equivalent,
	and in particular $\rho_{1}$ admits almost invariant vectors. This
	is a contradiction. 
\end{proof}
\begin{prop}
	\label{prop:amenable traces are bi-amenable}Let $\varphi$ be a trace
	of $\Lambda$, and let $\pi$ and $\rho$ be the left and right GNS
	representations corresponding to $\varphi$. If $\varphi$ is amenable
	then $\pi\times\rho:\Lambda\times\Lambda\to\mathcal{U}(\mathcal{H})$
	is amenable. 
\end{prop}

\begin{proof}
	The function $|\varphi|^{2}=\varphi\cdot\bar{\varphi}$ is a trace
	on $\Lambda$. Indeed, $\left(\mathcal{H}\otimes\mathcal{H},\pi\otimes\bar{\pi},\xi\otimes\xi\right)$ is the corresponding GNS triplet, and $\rho\otimes\bar{\rho}$ is the right GNS. As $\varphi$ is amenable,
	$\pi\otimes\bar{\pi}$ has almost invariant vectors, and so by Proposition \ref{prop:almost-inv-vectors-of-traces}
	so does $\left(\pi\otimes\bar{\pi}\right)\times\left(\rho\otimes\bar{\rho}\right)$.
	But the latter is canonically equivalent to $\left(\pi\times\rho\right)\otimes\overline{\left(\pi\times\rho\right)}$
	so that $\pi\times\rho$ is amenable. 
\end{proof}

\subsection{Invariant random subgroups and p.m.p actions } \label{subsec:IRS}

For a group $\Lambda$ denote by $\Sub({\Lambda})$ the space of all subgroups of $\Lambda$ endowed with the Chabauty topology. This topology is the one induced by the natural embedding $\Sub({\Lambda})\subseteq\{0,1\}^\Lambda$ with the product topology, and in particular it is compact. We say that $H\in \Sub{(\Lambda)}$ is \emph{co-amenable} if $\Lambda/H$ admits a $\Lambda$-invariant mean (see \cite{hartman2016stabilizer}).

\begin{lem}
	\label{lem:co-amenable-subgroups-are-Borel}The set of co-amenable
	subgroups of a countable group $\Lambda$ is Borel.
\end{lem}

\begin{proof}
	Let $F_{k}$ be an enumeration of all finite subsets of $\Lambda$.
	For a subgroup $H\leq\Lambda$ and a set $S\subseteq\Lambda$ let
	$S/H$ denote the image of $S$ under the quotient map $\Lambda\to\Lambda/H$.
	Clearly, $\left\{ F_{k}/H\right\} _{k\in\mathbb{N}}$ exhausts all
	possible finite subsets of $\Lambda/H$. For $g\in\Lambda$ and $k\in\N$
	define the map $f_{k,g}:\Sub({\Lambda})\to\Q$ defined by:
	\[
	f_{k,g}\left(H\right)=\frac{\left|\left(F_{k}\bigtriangleup g.F_{k}\right)/H\right|}{\left|F_{k}/H\right|}
	\]
	Suppose a sequence of subgroups $H_{n}$ converges to $H$. Since $\Lambda$
	is discrete, this just means that a given $x\in\Lambda$ is in $H$
	if and only if it is eventually in $H_{n}$. Therefore, $\left|S/H_{n}\right|$
	eventually coincides with $|S/H|$ for any given finite set $S\subseteq\Lambda$. 
	This shows that $f_{k,g}$ is locally constant and in particular
	continuous. It follows that the set
	\[
	\bigcap_{i,j\in\N}\bigcup_{k\in\N}\bigcap_{g\in F_{i}}f_{k,g}^{-1}\left([0,2^{-j})\right)
	\]
	is Borel. It is left to note that this set is precisely the set of
	all co-amenable subgroups, see the F\o lner condition for co-amenability \cite[Theorem 2.9]{glasner2007amenable}.
\end{proof}

We consider the continuous action of  $\Lambda$ on $\Sub(\Lambda)$ by conjugation. An \emph{invariant random subgroup} (IRS in short) of $\Lambda$ is a $\Lambda$-invariant Borel probability measure on $\Sub(\Lambda)$. A \emph{uniformly recurrent subgroup} (URS in short) is a minimal (non-empty) $\Lambda$-invariant closed subset of 
$\Sub({\Lambda})$. We refer to \cite{abert2014kesten,glasner2015uniformly} for more on these notions. 

By a \emph{p.m.p action} we mean a probability measure space $(\Omega,\nu)$ endowed with a measure preserving action of $\Lambda$. We say that such an action has a \emph{spectral gap} if the associated Koopman representation $\Lambda\to \mathcal{U}(L^2_0(\Omega,\nu))$ (where $L^2_0$ is the orthogonal complement of the constants) has no almost invariant vectors. We note that having a spectral gap implies ergodicity. The pushforward of $\nu$ under the stabilizer map $\omega\mapsto \mathrm{stab}_\Lambda(\omega)$ is an IRS on $\Lambda$. In fact, any IRS arises in this manner, and any ergodic IRS arises from an ergodic p.m.p action \cite{abert2014kesten}. 

\begin{thm}[\protect{Compare with \cite[Theorem 3.2]{peterson2016character} and with \cite[Theorem 2.12]{dudko2014finite}}] \label{thm:charm-dichotomy-IRS}
	Let $\Lambda$ be a countable group, and suppose  that any trace of $\Lambda$ is either amenable or supported on the amenable radical. Then for any p.m.p action $(\Omega,\nu)$ of $\Lambda$ with spectral gap, the stabilizer of a random point $\omega\in \Omega$ is either contained in $\Rad(\Lambda)$ almost surely, or it is co-amenable almost surely.  Thus, denoting by $\mu$ the corresponding IRS, either the set of all subgroups of $\Rad(\Lambda)$ is of full $\mu$-measure, or the set of all co-amenable subgroups is of full $\mu$-measure.  
\end{thm}

We shall first need some preparation.

\begin{lem}\label{lem:almost-inv-vectors}
	Consider a measure space $(\Omega,\nu)$ and a measure preserving action of a group $\Lambda$ on it, and fix $1\leq p<\infty$. Let $\mathcal{F}_p$ denote the set of functions $\Omega\to \R_{\geq 0}$ in the unit sphere of $L^p(\Omega,\nu)$.  Let $\mathcal{P}$ denote the set of probability measure on $\Omega$ which are absolutely continuous w.r.t $\nu$, endowed with the total variation metric.  Then the mapping:
	\[
		\phi_p:\mathcal{F}_p \to \mathcal{P}:  f\mapsto f^p\nu 
	\] 
	is a $\Lambda$-equivariant bijection, and both $\phi_p$ and $\phi^{-1}_p$ are uniformly continuous. In particular, for a sequence $f_n\in \mathcal{F}_p$ the following are equivalent:
	\begin{enumerate}
		\item $f_n$ are almost invariant in the $L^p$-norm: 
		\[
			\forall g\in \Lambda: \|g.f_n-f_n\|_p\to 0
		\]
		\item $\nu_n=\phi_p(f_n)$ are almost invariant in the total variation norm: 
		\[
			\forall g\in \Lambda: \|g.\nu_n-\nu_n\|_{\mathrm{tot}}\to 0
		\]
	\end{enumerate} 
\end{lem}

\begin{proof} 
	Consider first the map $\mathcal{F}_p \to \mathcal{F}_1: f\mapsto f^p$. This map is referred to in the literature as the Mazur map, and it is well known that both this map, as well as its inverse, are uniformly continuous \cite[Chapter 9.1]{benyamini1998geometric}. It is also clearly a $\Lambda$-equivariant bijection. We may therefore assume that $p=1$ and consider the map $\phi_1$.
	
	$\phi_1$ is of course bijective and it is evidently $\Lambda$-equivariant. It is a standard fact that the total variation distance between two absolutely continuous probability measures is precisely half of the $L^1$-distance of the corresponding probability density functions, so in particular,  $\phi_1$ and $\phi_1^{-1}$ are uniformly continuous.
	
	The equivalence between almost invariance of the $L^p$-functions and the corresponding measures follows from equivariance and uniform continuity.

\end{proof}

Let $\Lambda\curvearrowright\left(\Omega,\nu\right)$ be a p.m.p action.
We define $\varphi:\Lambda\to \C$
by $\varphi(g)=\nu\left(\Fix({g})\right)$ where $\Fix(g)=\left\{ \omega\in\Omega\mid g.\omega=\omega\right\}$,
and note that this is a trace on $\Lambda$. 
We recall a construction of a representation associated with this
p.m.p action that utilizes the GNS representation of $\varphi$.
Consider the orbit equivalence relation:
\[
\mathcal{R}:=\left\{ \left(g.\omega,\omega\right)\mid\omega\in\Omega,g\in\Lambda\right\} \subseteq\Omega\times\Omega
\]
endowed with the subspace $\sigma$-algebra of the product. Denote
by $p_{1}$ the projection of $\mathcal{R}$ onto the
first coordinate, and define the 
measure $\tilde{\nu}$ on $\mathcal{R}$ by setting for any measurable
$E\subseteq\mathcal{R}$:
\[
\tilde{\nu}\left(E\right)=\int_{\Omega}\left|p_{1}^{-1}(\omega)\cap E\right|d\nu
\]
In other words, $\tilde{\nu}$ is obtained by integrating the counting
measures on the fibers of $p_{1}$ along $\nu$. 

$\Lambda$ acts on $\mathcal{R}$ via the left coordinate, as well as via the right coordinate, and consider the corresponding Koopman representations
$\pi, \rho:\Lambda\to\mathcal{U}\left(L^{2}(\mathcal{R},\tilde{\nu})\right)$. $\pi$ and $\rho$ obviously commute with each other and satisfy $\pi(g)\rho(g)\boldsymbol{1}_{\Delta}=\boldsymbol{1}_{\Delta}$ where $\boldsymbol{1}_{\Delta}$ is the characteristic function of the diagonal of $\Omega\times\Omega$. 
Moreover $\varphi(g)=\left\langle \pi(g)\boldsymbol{1}_{\Delta},\boldsymbol{1}_{\Delta}\right\rangle $ where $\boldsymbol{1}_{\Delta}$ is the characteristic function of the diagonal of $\Omega\times\Omega$. This means the left and right GNS representations of $\varphi$ are the cyclic
subrepresentation of $\pi$ generated by $\boldsymbol{1}_{\Delta}$. See  \cite[Theorem 15.F.4]{bekka2019unitary} for further details.

\begin{lem}
	\label{lem:Bekka-amemable-co-amenable}Let $(\Omega,\nu)$ be a p.m.p action and let $\varphi:g\mapsto \nu(\Fix(g))$ be the corresponding trace of $\Lambda$. If $(\Omega,\nu)$  has a spectral gap
	and $\varphi$ is amenable then the stabilizer of $\nu$-a.e point in
	$\Omega$ is co-amenable. 
\end{lem}

\begin{proof}
	Consider the representation $\pi\times\rho:\Lambda\times\Lambda\to \mathcal{U}(L^2(\mathcal{R},\tilde{\nu}))$ where $(\mathcal{R},\tilde{\nu})$ is the orbit equivalence relation space associated with $(\Omega,\nu)$, as constructed above. By Proposition \ref{prop:amenable traces are bi-amenable}, the combined left and right GNS representation $\pi_\varphi\times\rho_\varphi$ of $\varphi$ is amenable. Hence  $\pi\times\rho$ is amenable as it contains $\pi_\varphi\times \rho_\varphi$. This means that there are functions $f_n\in L^2(\mathcal{R})^{\otimes 2}\subseteq L^2(\mathcal{R}^2)$ of norm $1$ which are almost invariant for the $\Lambda\times\Lambda$-action.   In particular $|f_n|$ are almost invariant and so by Lemma \ref{lem:almost-inv-vectors}, the probability measures $|f_n|^2 \tilde{\nu}^2$ on $\mathcal{R}^2$ are almost invariant (in the total variation norm). Thus so are the pushforward probability measure $\eta_n$ under the projection $\mathcal{R}^2\to \mathcal{R}$ (onto say, the first coordinate).

	The projection $p_1:\mathcal{R}\to \Omega$ onto the first coordinate is $\Lambda\times\{e\}$-equivariant. Thus the pushforward probability measures $\nu_{n}:={p_1}_*\eta_n$ on $\Omega$ are absolutely continuous w.r.t $\nu$ and are almost invariant. By slightly changing $\nu_n$ (e.g taking a convex combination with $\nu$) we may assume that they are all in the same measure class of $\nu$.  Using Lemma \ref{lem:almost-inv-vectors} once again, we deduce that the $L^2$-functions $\sqrt{{d\nu_n}/{d\nu}}$ are almost invariant.  
	By assumption $\Lambda\curvearrowright\left(\Omega,\nu\right)$
	has a spectral gap so it must be that $\sqrt{{d\nu_n}/{d\nu}}$ converge to $1$ in $L^2$,  which means that (again by Lemma \ref{lem:almost-inv-vectors}) $\nu_{n}$ converge to $\nu$ in the total variation norm. 
	
	We shall from now on consider only the action of $\Lambda$ on $\mathcal{R}$ via the second coordinate, that is the action of $\{e\}\times\Lambda$. This action preserves the fibers of $p_1$.  Let $\eta_{n}=\int_{\Omega}\eta_{n}^{\omega}d\nu_{n}$ be
	the disintegration of $\eta_{n}$ along $p_1$, and
	set $\eta_{n}'=\int_{\Omega}\eta_{n}^{\omega}d\nu$. Then
	$\|\eta_{n}'-\eta_{n}\|_\mathrm{tot} \to 0$ and therefore by the triangle inequality $\eta'_{n}$
	are almost invariant. 
	
	Choose an enumeration $\Lambda=\left\{ g_{i}\right\} _{i\in\N}$
	and let $F_{n}=\left\{ g_{1},...,g_{n}\right\} $. By perhaps replacing
	$\eta_{n}'$ with a subsequence, we may assume that $\max_{g\in F_{n}}\|g.\eta_{n}'-\eta_{n}'\|<5^{-n}$.
	For each $n\in \N$, the measures $\eta_n^\omega$ are supported on pairwise distinct fibers of $p_1$ and as a result: 
	\[
	\int_\Omega\|g.\eta_{n}^{\omega}-\eta_{n}^{\omega}||d\nu(\omega)=
	||\int_\Omega \left(g.\eta_{n}^{\omega}-\eta_{n}^{\omega}\right)d\nu(\omega)||=
	||g.\eta_{n}'-\eta_{n}'||<5^{-n},\,\left(\forall g\in F_{n}\right)
	\]
	Let $A_{n}:=\left\{ \omega\in\Omega\mid\max_{g\in F_n}\|g.\eta_{n}^{\omega}-\eta_{n}^{\omega}\|\geq2^{-n}\right\} $.
	Then by Markov's inequality $\nu\left(A_{n}\right)<\left(\frac{2}{5}\right)^{n}$,
	and therefore $\Omega_{0}=\bigcap_{n\in\mathbb{N}}\Omega\backslash A_{n}$
	has positive measure. For any $\omega\in\Omega_{0}$, $\eta_{n}^{\omega}$
	is a sequence of almost invariant measures on the orbit $\Lambda/\Lambda_{\omega}$ of $\omega$,
	and as a result, $\Lambda_{\omega}$ is co-amenable \cite{eymard1972moyennes}. It follows
	that the set of $\omega\in\Omega$ for which $\Lambda_{\omega}$ is
	co-amenable is a Borel (Lemma \ref{lem:co-amenable-subgroups-are-Borel})
	$\Lambda$-invariant set of positive $\nu$-measure, and so by ergodicity
	is has full measure. 
\end{proof}

\begin{proof}[Proof of Theorem \ref{thm:charm-dichotomy-IRS}]
	Let $\left(\Omega,\nu\right)$ be a p.m.p action of $\Lambda$ with spectral gap. Let $\varphi:g\mapsto\nu(\Fix (g))$ be
	the corresponding trace, and note that $\varphi(g)=\mu\left(H\leq\Lambda\mid g\in H\right)$
	where $\mu$ is the IRS of $\Lambda$ obtained by pushing forward
	$\nu$ under the stabilizer map $\Omega\to\Sub({\Lambda})$.
	By assumption, either $\varphi$ is amenable or it is supported on the amenable
	radical. If the former case occurs then $\mu$-a.e subgroup is co-amenable by Lemma \ref{lem:Bekka-amemable-co-amenable}. The latter case means that $\mu\left(H\leq\Lambda\mid g\notin H\right)=1-\varphi(g)=1$
	for any $g\in\Lambda\backslash\Rad({\Lambda})$. Thus, as $\Lambda$
	is countable:
	\[
	\mu\left(H\leq\Lambda\mid H\subseteq\Rad({\Lambda})\right)=\mu\left(\bigcap_{g\in\Lambda\backslash \Rad(\Lambda)}\left\{ H\leq\Lambda\mid g\notin H\right\} \right)=1
	\]
	meaning that $\mu$-a.e subgroup is contained in $\Rad(\Lambda)$.  
	
\end{proof}

\subsection{Representations and their $C^*$-algebras}
The \emph{$C^*$-algebra of a representation $\pi:\Lambda\to \mathcal{U}(\mathcal{H})$} is by definition the $C^*$-subalgebra $C^*_\pi(\Lambda)\subseteq \mathcal{B}(\mathcal{H})$ generated by the image of $\pi$.

\begin{prop}\label{prop:charmenable-reps}
	Let $\Lambda$ be a charmenable group and let $\pi$ be a non-amenable
	representation. Denote by $\mathcal{S}$ the collection of all 
	representations of $\Lambda$ that are induced from a representation of $\Rad({\Lambda})$ and that are weakly contained in
	$\pi$. Then:
	\begin{enumerate}
		\item $\mathcal{S}$ is not empty. 
		\item For any $\rho\in\mathcal{S}$, the assignment $\pi(g)\mapsto\rho(g)$
		extends to a unital surjective $*$-homomorphism $\Theta_{\rho}:C_{\pi}^{*}(\Lambda)\to C_{\rho}^{*}(\Lambda)$. 
		\item Any maximal ideal of $C_{\pi}^{*}(\Lambda)$ is the kernel of $\Theta_{\rho}$
		for some $\rho\in\mathcal{S}$. 
		\item Any tracial state on $C_{\pi}^{*}(\Lambda)$ factorizes through $\Theta_{\rho}$
		for some $\rho\in\mathcal{S}$. 
	\end{enumerate}
\end{prop}
This proposition is analogous to  \cite[Proposition 3.6, 3.7]{bader2022charmenability}, and the proof is in spirit the same. For completeness, we write out the details.  
\begin{proof}
	Let $C\subseteq\PD_1({\Lambda})$ be the collection of all positive-definite
	functions of the form $\phi\circ\pi$ for a state $\phi$ on $C_{\pi}^{*}(\Lambda)$.
	This is a closed convex $\Lambda$-invariant subset.
	By charmenability it contains at least one trace $\varphi$. The GNS representation
	$\pi_{\varphi}$ of $\varphi$ is weakly contained in $\pi$ and therefore
	cannot be amenable. This means $\varphi$ is not amenable and so by Proposition \ref{prop:trace-dichotomy}, it is
	supported on the amenable radical.  
	By Lemma \ref{lem:induction-trivial-extension}, $\pi_{\varphi}$
	is induced from the amenable radical, i.e $\pi_{\varphi}\in\mathcal{S}$. This proves
	1, and 2 follows as it is one of the equivalent definitions of weak 	containment of representations \cite[Proposition 8.B.4]{bekka2019unitary}.

	For 3, let $I$ be a proper ideal of $C_{\pi}^{*}(\Lambda)$ with
	quotient map $q:C_{\pi}^{*}(\Lambda)\to C_{\pi}^{*}(\Lambda)/I$ and choose some faithful $*$-embedding $C_{\pi}^{*}(\Lambda)/I\subseteq \mathcal{B}(\mathcal{H}_0)$.
	The representation $q\circ\pi$ of $\Lambda$ is not amenable
	as it is weakly contained in $\pi$. By the above, there exists a representation $\rho_{0}\in \mathcal{S}$
	and a unital surjective
	$*$-homomorphisms $\sigma:C_{\pi}^{*}(\Lambda)/I\to C_{\rho_{0}}^{*}(\Lambda)$.
	In particular, $\Theta_{\rho_{0}}=\sigma\circ q$ is a unital surjective
	$*$-homomorphism and $I$ is contained
	in the kernel of $\Theta_{\rho_{0}}$. 
	
	Finally for 4, let $\phi$ be a tracial state of $C_{\pi}^{*}(\Lambda)$, let $\varphi:=\phi\circ\pi\in \Tr(\Lambda)$, and let  $\pi_{\varphi}$  be
	the GNS representation associated to $\varphi$. Then $\phi$ factorizes
	through $\Theta_{\pi_{\varphi}}:C_{\pi}^{*}(\Lambda)\to C_{\pi_{\varphi}}^{*}(\Lambda)$.
	As $\pi_{\varphi}$ is weakly contained in $\pi$ it cannot be amenable, so $\varphi$ must be
	supported on $\Rad({\Lambda})$ (Proposition \ref{prop:trace-dichotomy}). It follows from Lemma \ref{lem:induction-trivial-extension} that $\pi_{\varphi}$
	is induced from the amenable radical, i.e $\pi_{\varphi}\in\mathcal{S}$.
\end{proof}

\subsection{Concluding the proofs}

\begin{proof}[Proof of Proposition \ref{prop:charmenable dichotomy}]
	(\ref{charm-item-normal}), (\ref{charm-item-trace}) and (\ref{charm-item-URS}) are respectively established
	in \cite[Propositions 3.3, 3.2, 3.5]{bader2022charmenability}. (\ref{charm-item-vn}) is an immediate consequence of Theorem \ref{thm:thoma correspondence} and in light of the classical fact that all amenable finite factors are hyperfinite \cite[Chapter 11]{anantharaman2017introduction}. 
	(\ref{charm-item-unirep}) and (\ref{charm-item-C*}) follow from  Proposition \ref{prop:charmenable-reps}.   (\ref{charm-item-IRS}) follows from (\ref{charm-item-trace}) due to Theorem \ref{thm:charm-dichotomy-IRS}.
\end{proof}	

Note that the second condition of charmenability alone implies (\ref{charm-item-normal}), (\ref{charm-item-trace}), (\ref{charm-item-IRS}). 

\begin{proof}[Proof of Proposition \ref{prop:charmenable dichotomy prop (T)}]
	We start with the statement about IRS. Let $X$ be an ergodic p.m.p action of $\Lambda$, and let $\mu$ be the corresponding IRS. Since $\Lambda$ has property (T), $X$ has a spectral gap. We may therefore apply Theorem \ref{thm:charm-dichotomy-IRS}. If $\mu$-a.e subgroup is contained in $\Rad(\Lambda)$ then it is finitely supported as it is ergodic and $\Sub(\Rad(\Lambda))$ is countable. If $\mu$-a.e subgroup is co-amenable, then $\mu$-a.e subgroup is in fact of finite index as $\Lambda$ has property (T) \cite[Claim 5.9]{hartman2016stabilizer}.
	
	Now for the statement about URS, let $X$ be a URS of $\Lambda$ and apply Proposition \ref{prop:charmenable dichotomy} (\ref{charm-item-URS}). If  $X$ admits an invariant probability measure then it in particular admits an ergodic one, that is an ergodic IRS. We may therefore apply the above to conclude that this IRS is supported of a finite orbit which must be equal to $X$ by minimality. Suppose now that $X\subseteq \Sub(\Rad(\Lambda))$. The derived space $X'$, i.e. the set of accumulation points of the compact space $X$, is a closed subset of $X$ which is invariant under all homeomorphisms. By minimality, $X'$ is either empty or all $X$. The latter is not possible since perfect Polish spaces have the cardinality of the continuum.  It follows that $X$ is discrete and compact and therefore finite.
\end{proof}

\bibliographystyle{alpha}
\bibliography{ItivigiBib}

\begin{thebibliography}{BFLM11}

\bibitem[AGV14]{abert2014kesten}
Mikl{\'o}s Ab{\'e}rt, Yair Glasner, and B{\'a}lint Vir{\'a}g.
\newblock Kesten’s theorem for invariant random subgroups.
\newblock {\em Duke Mathematical Journal}, 163(3):465--488, 2014.

\bibitem[AP17]{anantharaman2017introduction}
Claire Anantharaman and Sorin Popa.
\newblock An introduction to $\mathrm{II}_1$ factors.
\newblock {\em preprint}, 8, 2017.

\bibitem[BBH21]{bader2021charmenability}
Uri Bader, R{\'e}mi Boutonnet, and Cyril Houdayer.
\newblock Charmenability of higher rank arithmetic groups.
\newblock {\em arXiv preprint arXiv:2112.01337}, 2021.

\bibitem[BBHP22]{bader2022charmenability}
Uri Bader, R{\'e}mi Boutonnet, Cyril Houdayer, and Jesse Peterson.
\newblock Charmenability of arithmetic groups of product type.
\newblock {\em Inventiones mathematicae}, pages 1--57, 2022.

\bibitem[BdlH19]{bekka2019unitary}
Bachir Bekka and Pierre de~la Harpe.
\newblock Unitary representations of groups, duals, and characters.
\newblock {\em arXiv preprint arXiv:1912.07262}, 2019.

\bibitem[Bek90]{bekka1990amenable}
Mohammed~EB Bekka.
\newblock Amenable unitary representations of locally compact groups.
\newblock {\em Inventiones mathematicae}, 100(1):383--401, 1990.

\bibitem[Bek07]{bekka2007operator}
Bachir Bekka.
\newblock Operator-algebraic superridigity for $\mathrm{SL}_n(\mathbb{Z}), n
  \geq 3$.
\newblock {\em Inventiones mathematicae}, 169(2):401--425, 2007.

\bibitem[Bek21]{bekka2021plancherel}
Bachir Bekka.
\newblock The plancherel formula for countable groups.
\newblock {\em Indagationes Mathematicae}, 32(3):619--638, 2021.

\bibitem[BF14]{bader2014boundaries}
Uri Bader and Alex Furman.
\newblock Boundaries, rigidity of representations, and lyapunov exponents.
\newblock {\em arXiv preprint arXiv:1404.5107}, 2014.

\bibitem[BF20]{bekka2020characters}
Bachir Bekka and Camille Francini.
\newblock Characters of algebraic groups over number fields.
\newblock {\em arXiv preprint arXiv:2002.07497}, 2020.

\bibitem[BFLM11]{BFLM2011stiffness}
Jean Bourgain, Alex Furman, Elon Lindenstrauss, and Shahar Mozes.
\newblock Stationary measures and equidistribution for orbits of nonabelian
  semigroups on the torus.
\newblock {\em Journal of the American Mathematical Society}, 24(1):231--280,
  2011.

\bibitem[BH21]{boutonnet2021stationary}
R{\'e}mi Boutonnet and Cyril Houdayer.
\newblock Stationary characters on lattices of semisimple lie groups.
\newblock {\em Publications math{\'e}matiques de l'IH{\'E}S}, 133(1):1--46,
  2021.

\bibitem[BL98]{benyamini1998geometric}
Yoav Benyamini and Joram Lindenstrauss.
\newblock {\em Geometric nonlinear functional analysis}, volume~48.
\newblock American Mathematical Soc., 1998.

\bibitem[BQ13]{benoist2013stationary}
Yves Benoist and Jean-Fran{\c{c}}ois Quint.
\newblock Stationary measures and invariant subsets of homogeneous spaces (ii).
\newblock {\em Journal of the American Mathematical Society}, 26(3):659--734,
  2013.

\bibitem[BS06]{bader2006factor}
Uri Bader and Yehuda Shalom.
\newblock Factor and normal subgroup theorems for lattices in products of
  groups.
\newblock {\em Inventiones mathematicae}, 163(2):415--454, 2006.

\bibitem[CM84]{carey-moran1984nil-char}
Alan~L Carey and William Moran.
\newblock Characters of nilpotent groups.
\newblock In {\em Mathematical Proceedings of the Cambridge Philosophical
  Society}, volume~96, pages 123--137. Cambridge University Press, 1984.

\bibitem[Con76]{connes1976classification}
Alain Connes.
\newblock Classification of injective factors.
\newblock {\em Annals of Mathematics}, pages 73--115, 1976.

\bibitem[DM14]{dudko2014finite}
Artem Dudko and Konstantin Medynets.
\newblock Finite factor representations of higman--thompson groups.
\newblock {\em Groups, Geometry, and Dynamics}, 8(2):375--389, 2014.

\bibitem[Eym72]{eymard1972moyennes}
Pierre Eymard.
\newblock Moyennes invariantes et representations des produits semi-directs.
\newblock In {\em Moyennes Invariantes et Repr{\'e}sentations Unitaires}, pages
  60--82. Springer, 1972.

\bibitem[Fur67]{furstenberg1967poisson}
Harry Furstenberg.
\newblock Poisson boundaries and envelopes of discrete groups.
\newblock {\em Bulletin of the American Mathematical Society}, 73(3):350--356,
  1967.

\bibitem[Fur98]{furstenberg1998stiffness}
Hillel Furstenberg.
\newblock Stiffness of group actions.
\newblock {\em Lie groups and ergodic theory (Mumbai, 1996)}, 14:105--117,
  1998.

\bibitem[GM07]{glasner2007amenable}
Yair Glasner and Nicolas Monod.
\newblock Amenable actions, free products and a fixed point property.
\newblock {\em Bulletin of the London Mathematical Society}, 39(1):138--150,
  2007.

\bibitem[GW15]{glasner2015uniformly}
Eli Glasner and Benjamin Weiss.
\newblock Uniformly recurrent subgroups.
\newblock {\em Recent trends in ergodic theory and dynamical systems},
  631:63--75, 2015.

\bibitem[HLT14]{hartman2014abramov}
Yair Hartman, Yuri Lima, and Omer Tamuz.
\newblock An abramov formula for stationary spaces of discrete groups.
\newblock {\em Ergodic Theory and Dynamical Systems}, 34(3):837--853, 2014.

\bibitem[Hou21]{houdayer2021noncommutative}
Cyril Houdayer.
\newblock Noncommutative ergodic theory of higher rank lattices.
\newblock {\em arXiv preprint arXiv:2110.07708}, 2021.

\bibitem[How77]{howe1977representations}
Roger Howe.
\newblock On representations of discrete, finitely generated, torsion-free,
  nilpotent groups.
\newblock {\em Pacific Journal of Mathematics}, 73(2):281--305, 1977.

\bibitem[HT16]{hartman2016stabilizer}
Yair Hartman and Omer Tamuz.
\newblock Stabilizer rigidity in irreducible group actions.
\newblock {\em Israel Journal of Mathematics}, 216(2):679--705, 2016.

\bibitem[Kai00]{Kaimanovich2000Poisson}
Vadim~A. Kaimanovich.
\newblock The {P}oisson formula for groups with hyperbolic properties.
\newblock {\em Ann. of Math. (2)}, 152(3):659--692, 2000.

\bibitem[Kai02]{kaimanovich2002extensions}
Vadim~A. Kaimanovich.
\newblock The {P}oisson boundary of amenable extensions.
\newblock {\em Monatsh. Math.}, 136(1):9--15, 2002.

\bibitem[Kai03]{kaimanovich2003double}
Vadim~A Kaimanovich.
\newblock Double ergodicity of the poisson boundary and applications to bounded
  cohomology.
\newblock {\em Geometric \& Functional Analysis GAFA}, 13(4):852--861, 2003.

\bibitem[Kan80]{kaniuth1980ideals}
Eberhard Kaniuth.
\newblock Ideals in group algebras of finitely generated fc-nilpotent discrete
  groups.
\newblock {\em Mathematische Annalen}, 248(2):97--108, 1980.

\bibitem[Kan06]{kaniuth2006induced}
Eberhard Kaniuth.
\newblock Induced characters, mackey analysis and primitive ideal spaces of
  nilpotent discrete groups.
\newblock {\em Journal of Functional Analysis}, 240(2):349--372, 2006.

\bibitem[KS21]{kennedy2021noncommutative}
Matthew Kennedy and Eli Shamovich.
\newblock Noncommutative choquet simplices.
\newblock {\em Mathematische Annalen}, pages 1--39, 2021.

\bibitem[LV22]{levit2022characters}
Arie Levit and Itamar Vigdorovich.
\newblock Characters of solvable groups, hilbert-schmidt stability and dense
  periodic measures.
\newblock {\em arXiv preprint arXiv:2206.02268}, 2022.

\bibitem[Mal51]{mal1951certain}
AI~Mal’cev.
\newblock On certain classes of infinite soluble groups.
\newblock {\em Mat. Sb}, 28(567-588):4, 1951.

\bibitem[Mar91]{margulis1991discrete}
Gregori~A Margulis.
\newblock {\em Discrete subgroups of semisimple Lie groups}, volume~17.
\newblock Springer Science \& Business Media, 1991.

\bibitem[MN36]{murray1936rings}
Francis~J Murray and J~v Neumann.
\newblock On rings of operators.
\newblock {\em Annals of Mathematics}, pages 116--229, 1936.

\bibitem[Mor01]{morris2001introduction}
Dave~Witte Morris.
\newblock Introduction to arithmetic groups.
\newblock {\em arXiv preprint math/0106063}, 2001.

\bibitem[Moz95]{mozes1995epimorphic}
Shahar Mozes.
\newblock Epimorphic subgroups and invariant measures.
\newblock {\em Ergodic Theory and Dynamical Systems}, 15(6):1207--1210, 1995.

\bibitem[OSV]{poulsen22freegroup}
Joav Orovitz, Raz Slutsky, and Itamar Vigdorovich.
\newblock The trace simplex of the free group and free product c$^*$-algebras.
\newblock {\em preprint}.

\bibitem[Pet13]{peterson2013notes}
Jesse Peterson.
\newblock Notes on von neumann algebras.
\newblock {\em Vanderbilt University}, 2013.

\bibitem[Pet14]{peterson2014character}
Jesse Peterson.
\newblock Character rigidity for lattices in higher-rank groups.
\newblock {\em Preprint.(http://www. math. vanderbilt. edu/peters10/rigidity.
  pdf)}, 2014.

\bibitem[Phe01]{phelps2001lectures}
Robert~R Phelps.
\newblock {\em Lectures on Choquet's theorem}.
\newblock Springer Science \& Business Media, 2001.

\bibitem[PRR93]{platonov-rapinchik-1993}
Vladimir Platonov, Andrei Rapinchuk, and Rachel Rowen.
\newblock {\em Algebraic groups and number theory}.
\newblock Academic press, 1993.

\bibitem[PT16]{peterson2016character}
Jesse Peterson and Andreas Thom.
\newblock Character rigidity for special linear groups.
\newblock {\em Journal f{\"u}r die reine und angewandte Mathematik (Crelles
  Journal)}, 2016(716):207--228, 2016.

\bibitem[Rag72]{raghunathan1972discrete}
Madabusi~Santanam Raghunathan.
\newblock {\em Discrete subgroups of Lie groups}, volume~3.
\newblock Springer, 1972.

\bibitem[Rob72]{finiteness-soluble-groups}
Derek~J.S. Robinson.
\newblock {\em Finiteness conditions and generalized soluble groups: Part 1}.
\newblock Berlin, Heidelberg, New York: Springer, 1972.

\bibitem[Tho64]{thoma1964unitare}
Elmar Thoma.
\newblock {\"U}ber unit{\"a}re darstellungen abz{\"a}hlbarer, diskreter
  gruppen.
\newblock {\em Mathematische Annalen}, 153(2):111--138, 1964.

\bibitem[Zim13]{zimmer2013ergodic}
Robert~J Zimmer.
\newblock {\em Ergodic theory and semisimple groups}, volume~81.
\newblock Springer Science \& Business Media, 2013.

\bibitem[Zim20]{zimmer2020extensions}
Robert~J Zimmer.
\newblock A. extensions of ergodic group actions.
\newblock In {\em Group Actions in Ergodic Theory, Geometry, and Topology},
  pages 17--53. University of Chicago Press, 2020.

\end{thebibliography}

\end{document}